\documentclass[11pt]{amsart}
\usepackage{titletoc}
\usepackage{setspace}
\usepackage[margin=1in]{geometry}

\let\p\partial
\numberwithin{equation}{section}

\usepackage{amssymb,latexsym,amsfonts}
\usepackage{pb-diagram}
\usepackage{graphicx}
\usepackage{hyperref}
\usepackage{verbatim}
\usepackage{color} 
\usepackage[matrix,arrow]{xy}
\usepackage{marginnote}
\usepackage{tikz}

\usetikzlibrary{matrix}

\newcommand{\Ric}{\mathrm{Ric}}

\theoremstyle{plain}
\newtheorem{prop}{Proposition}
\newtheorem{lemma}{Lemma}
\newtheorem{theorem}{Theorem}
\newtheorem{cor}{Corollary}

\newtheorem*{theorem-main}{Main Theorem}

\theoremstyle{definition}

\newtheorem*{rem}{Remark}

\newtheorem*{theorem-z}{Theorem}

\newcommand{\Riem}{{\mathcal R}}
\newcommand{\Diff}{\mathrm{Diff}}
\newcommand{\im}{\mathrm{Im}\ }
\newcommand{\BDiff}{\mathrm{BDiff}}
\newcommand{\bM}{{\mathcal{M}}}
\newcommand{\Q}{{\mathbb Q}}
\newcommand{\M}{{\mathcal{M}}}
\begin{document}

\title [Observer Moduli space of Ricci positive metrics] {Homotopy groups
  of the observer moduli space of \\ Ricci positive metrics}

\begin{abstract} 
The observer moduli space of Riemannian metrics is the quotient of the
space $\Riem(M)$ of all Riemannian metrics on a manifold $M$ by the
group of diffeomorphisms $\Diff_{x_0}(M)$ which fix both a basepoint
$x_0$ and the tangent space at $x_0$. The group $\Diff_{x_0}(M)$ acts
freely on $\Riem(M)$ providing $M$ is connected. This offers certain
advantages over the classic moduli space, which is the quotient by the
full diffeomorphism group.  Results due to Botvinnik, Hanke, Schick
and Walsh, and to Hanke, Schick and Steimle have demonstrated that the
higher homotopy groups of the observer moduli space
$\mathcal{M}_{x_0}^{s>0}(M)$ of positive scalar curvature metrics are,
in many cases, non-trivial.  The aim in the current paper is to
establish similar results for the moduli space
$\mathcal{M}_{x_0}^{\Ric>0}(M)$ of metrics with positive Ricci
curvature. In particular we show that for a given $k$, there are
infinite order elements in the homotopy group
$\pi_{4k}\mathcal{M}_{x_0}^{\Ric>0}(S^n)$ provided the dimension $n$
is odd and sufficiently large. In establishing this we make use of a
gluing result of Perelman. We provide full details of the proof of
this gluing theorem, which we believe have not appeared before in the
literature. We also extend this to a family gluing theorem for Ricci
positive manifolds.
\end{abstract}

\author{Boris Botvinnik}
\address{
Department of Mathematics\\
University of Oregon \\
Eugene, OR, 97405\\
USA
}
\email{botvinn@uoregon.edu}

\author{Mark G. Walsh}
\address{Department of Mathematics and Statistics\\
National University of Ireland Maynooth\\
Maynooth\\
Ireland}
\email{mark.walsh@mu.ie}

\author{David J. Wraith}
\address{Department of Mathematics and Statistics\\
National University of Ireland Maynooth\\
Maynooth\\
Ireland}
\email{david.wraith@mu.ie}
\subjclass[2000]{53C27, 57R65, 58J05, 58J50}

\keywords{Positive Ricci metrics, Hatcher bundles, Moduli spaces.}
\date{\today}

\maketitle

\renewcommand{\baselinestretch}{1.1}\normalsize
\tableofcontents

\renewcommand{\baselinestretch}{1.3}\normalsize

\section{Introduction}\label{sec:Intro}
\subsection{Motivation and main result}
In recent years, there have been great efforts made to better
understand the topology of moduli spaces of Riemannian metrics of
positive scalar curvature on a smooth compact (usually spin) manifold;
see {\cite{Carr,BG,BHSW,HSS}}. Apart from results of Kreck and Stolz
in \cite{KS} and Wraith in \cite{Wraith} concerning path-connectivity,
we know very little about topology of the corresponding moduli spaces
of positive Ricci curvature metrics.  (In this context we should also
mention work of Dessai, Klaus and Tuschmann on moduli spaces of
non-negative sectional curvature metrics in \cite{DKT}, and the
results of Crowley, Schick and Steimle on the space of Ricci positive
metrics on certain manifolds, see \cite{CSS}.)  Whether or not there
is any non-triviality in the higher homotopy groups of such moduli
spaces is still an open question.  Here we study the topology of its
closest relative, the {\em observer moduli space}
$\M_{x_0}^{\Ric>0}(S^{n})$ of positive Ricci curvature metrics on the
sphere $S^{n}$.

We denote by $ds_{n}^{2}$ the standard round metric on $S^{n}$,
and by $[ds_{n}^{2}]$ its orbit in the moduli space
$\M_{x_0}^{\Ric>0}(S^{n})$. Here is our main result:
\begin{theorem-main}
For any $m\in \mathbb{N}$, there is an integer $N(m)$ such that for
all odd $n>N(m)$, the group $\pi_{i}(\M_{x_0}^{\Ric>0}(S^{n}),
[ds_{n}^{2}]) \otimes {\mathbb Q}$ is non-trivial when $i=4k$ and $k\leq m$.
\end{theorem-main}
We would like to emphasize that the observer moduli space is indeed
the most tractable moduli space of metrics. Let $\Riem(M)$ be the
space of all metrics on a compact closed manifold $M$, and $\Diff(M)$
be the group of diffeomorphisms which acts naturally on $\Riem(M)$ by
pull-back.  Even though the space $\Riem(M)$ is contractible, the
moduli space of all metrics, i.e. the orbit space $\Riem(M)/\Diff(M)$,
could be very complicated since some metrics have non-trivial isometry
groups. Hence, in general, the action of $\Diff(M)$ on the space of
metrics $\Riem(M)$ is far from being tractable. Following ideas from
Gauge Theory, we fix an observer, i.e. a base point $x_0\in M$
together with a frame at the tangent space $T_{x_0}M$. Then we obtain
the \emph{observer moduli space}
$\M_{x_0}(M):=\Riem(M)/\Diff_{x_0}(M)$, where the gauge group
$\Diff_{x_0}(M)$ fixes such an observer. It is easy to see that the
gauge group $\Diff_{x_0}(M)$ acts freely on the space of metrics
provided $M$ is a connected manifold. Then the observer moduli space
$\M_{x_0}(M)$ is homotopy equivalent to the classifying space
$\BDiff_{x_0}(M)$, and the corresponding observer moduli space
$\M_{x_0}^{\Ric>0}(M)$ of positive Ricci metrics maps naturally to
$\M_{x_0}(M)$, see below for more details.

The proof of Main Theorem is based on an analogous theorem by
Botvinnik, Hanke, Schick and Walsh for the observer moduli space of
positive {\em scalar} curvature metrics; see
\cite{BHSW}. Both proofs rely heavily on work of Farrell, Hsiang,
Hatcher and Goette; see \cite{FH} and \cite{GO}.  Techniques for
constructing families of metrics are also required. In the scalar
curvature case, this means a family version of the Gromov-Lawson
surgery technique from \cite{GL}, described in \cite{Walsh2}. Due to
the flexibility of the scalar curvature and the strength of the
Gromov-Lawson construction, this technique permits the detection of
non-triviality for manifolds besides the sphere. Unsurprisingly, the
Ricci curvature case requires a more delicate construction, which is based on a
gluing theorem of Perelman. As yet, we have not demonstrated
non-triviality beyond the case of the sphere.
\subsection{The observer moduli spaces of metrics}
Let $M$ be a smooth closed connected manifold of dimension $n$. We
denote by $\Riem(M)$, the space of all Riemannian metrics on $M$,
equipped with the smooth Whitney topology.  For a metric $g\in
\Riem(M)$, we denote by $s_g$ and $\Ric_g$ its scalar and Ricci
curvatures. We then consider the subspaces
\begin{equation*}
\Riem^{s>0}(M) \subset \Riem(M) \ \ \ \mbox{and} \ \ \ \Riem^{\Ric>0}(M)
\subset \Riem(M),
\end{equation*}
of metrics with positive scalar and positive Ricci curvatures
respectively.  Let $\Diff(M)$ be the group of diffeomorphisms on
$M$. This group acts on the space of metrics by pull-back:
\begin{equation*}
\Diff(M)\times \Riem(M) \to \Riem(M),  \ \ \ \ (\phi,g)\mapsto \phi^* g.
\end{equation*}
Recalling that $M$ is connected, we fix a base point $x_0\in M$ which
plays the role of an \emph{observer} in a sense which will become
clear shortly. Let $\Diff_{x_0}(M)\subset \Diff(M)$ be the subgroup of
diffeomorphisms $\phi : M\to M$ such that $\phi(x_0)=x_0$ and such
that the derivative $d\phi_{x_0}: T_{x_0}M\to T_{x_0}M$ is the
identity. This is the {\it observer diffeomorphism group} of $M$ based
at $x_0.$

As we have mentioned, the group $\Diff_{x_0}(M)$ acts freely on the
space of metrics $\Riem(M)$ provided $M$ is a connected manifold; see
\cite[Lemma 1.2]{BHSW}.  The orbit space $\bM_{x_0}(M):=
\Riem(M)/\Diff_{x_0}$ is the \emph{observer moduli space of metrics on
  $M$}.  Since the space $\Riem(M)$ is contractible and the action of
$\Diff_{x_0}(M)$ on $\Riem(M)$ is \emph{proper}, see \cite[Lemma
  1.2]{Ebin}, the observer moduli space $\bM_{x_0}(M)$ is homotopy
equivalent to the classifying space $\BDiff_{x_0}(M)$ of the group
$\Diff_{x_0}(M)$. In particular, we have a $\Diff_{x_0}(M)$-principal
bundle:
\begin{equation*}
\Diff_{x_0}(M)\to \Riem(M) \to \bM_{x_0}(M).
\end{equation*}  
By restricting the action of $\Diff_{x_0}(M)$ to the appropriate subspaces, we
obtain the \emph{observer moduli spaces}
\begin{equation*}
  \bM_{x_0}^{s>0}(M):= \Riem^{s>0}(M)/\Diff_{x_0}(M), \ \ \
  \bM_{x_0}^{\Ric>0}(M):= \Riem^{\Ric>0}(M)/\Diff_{x_0}(M),
\end{equation*} 
\emph{of positive scalar} and \emph{of positive Ricci curvature metrics}
respectively. The inclusions of spaces of metrics
$\Riem^{\Ric>0}(M)\subset \Riem^{s>0}(M)\subset \Riem(M)$ then induce the
maps of principal $\Diff_{x_0}(M)$-bundles:
\begin{equation}\label{eq:principal1}
\xymatrix{
  \Riem^{\Ric>0}(M) \ar[r] \ar[d] & \Riem^{s>0}(M) \ar[r] \ar[d]
  & \Riem (M) \ar[d]
  \\
  \bM_{x_0}^{\Ric>0}(M) \ar[r]^{\iota_1} & \bM_{x_0}^{s>0}(M) \ar[r]^{\iota_0} &
  \bM_{x_0}(M)   
}
\end{equation}   
We denote $\iota:= \iota_0\circ\iota_1:\bM_{x_0}^{\Ric>0}(M)\to
\bM_{x_0}(M)$. The fibre bundles
(\ref{eq:principal1}) give rise to the following
commutative diagram, where the horizontal lines
are Serre fibrations: 
\begin{equation}\label{eq:principal2}
\xymatrix{
  \Riem^{\Ric>0}(M) \ar[r] \ar[d] &  \bM_{x_0}^{\Ric>0}(M) \ar[r]^{\iota}
  \ar[d]^{\iota_1}
  & \bM_{x_0}(M) \ar[d]^{Id}
  \\
  \Riem^{s>0}(M) \ar[r] & \bM_{x_0}^{s>0}(M) \ar[r]^{\iota_0} &
  \bM_{x_0}(M) 
}
\end{equation}  
Letting $g_0$ denote a base point metric in $\Riem(M)$, we consider
the induced diagram of homotopy group homomorphisms below:
\begin{equation}\label{eq:principal3}
\xymatrix{
   \pi_{i}(\bM_{x_0}^{\Ric>0}(M), [g_0]) \ar[r]^{\iota_{*}}
  \ar[d]^{{\iota_1}_{*}}
  & \pi_{i}(\bM_{x_0}(M), [g_0]) \ar[d]^{Id}
  \\
   \pi_{i}(\bM_{x_0}^{s>0}(M), [g_0]) \ar[r]^{{\iota_0}_{*}} &
  \pi_{i}(\bM_{x_0}(M), [g_0])   
}
\end{equation}
It is well-known that an element in the homotopy group
$\pi_{i}(\bM_{x_0}(M), [g_0]) $ can be represented by a smooth fibre
bundle $E\to S^i$ with a fibre $M$. Hence to show that such an element
lies in the image of ${\iota_{0}}_{*}$, it is enough to show that
there exists a metric on the total space $E$ which restricts to a
psc-metric on every fibre, see \cite{BHSW}. Here our task is more
difficult: we have to construct such a metric on $E$ which is fibre-wise
Ricci-positive, and the methods used involve geometric
constructions which are quite different from the positive scalar
curvature case. This is one of the reasons why we restrict our
attention to the case when $M=S^n$.  Next, we focus on the geometrical
properties of the moduli space $\bM_{x_0}(M)$.

\subsection{The universal fibre metric}\label{universal-fibre}
As we have mentioned earlier, the observer moduli
  space $\bM_{x_0}(M)$ is homotopy equivalent to the classifying space
  $\BDiff_{x_0}(M)$.

We say that a fibre bundle $E\to X$ with fibre $M$ is a \emph{smooth
  $M$-fibre bundle} if its structure group is a subgroup of
$\Diff_{x_0}(M)$.  Now we consider the universal principal bundle $
\Riem(M)\rightarrow \M_{x_0}(M).  $ Here the group $\Diff_{x_0}(M)$
acts freely on $\Riem(M)$, and the Borel construction gives the
universal smooth $M$-fibre bundle $E(M)\to \M_{x_0}(M)$, where $E(M)
:= \Riem(M)\times_{\Diff_{x_0}(M)} M$.  Recall that the space
$\Riem(M)\times_{\Diff_{x_0}(M)} M$ is defined as the quotient of
$\Riem(M)\times M$ by the action of $\Diff_{x_0}(M)$ given by
$\phi.(h,x)=({(\phi^{-1})}^{*}h, \phi(x))$, where $\phi\in
\Diff_{x_0}(M)$, $h\in \Riem(M)$ and $x\in M$.

Given that $X$ is a paracompact Hausdorff space, recall that the
isomorphism classes of principal $\Diff_{x_0}(M)$-bundles over $X$ are
in one to one correspondence with homotopy classes $[X,\M_{x_0}(M)]$
of maps $X\rightarrow \M_{x_0}(M)$.  In particular, given a map
$f:X\rightarrow \M_{x_0}(M)$, we obtain a commutative diagram:
\begin{equation*}\label{eq:principal}
\xymatrix{
  E_f \ar[r] \ar[d] & E(M) \ar[d] 
  \\
 X \ar[r]^{f} & \M_{x_0}(M)}
\end{equation*}   
where the bundle $E_f\rightarrow X$ is the pull-back of the universal 
smooth $M$-fibre bundle by the map $f$.

There is however a more refined structure which we can associate to
such a bundle. The total space $E(M)=\Riem(M)\times_{\Diff_{x_0}(M)}
M$ admits a ``universal fibre metric" which we will now define. We
begin with an arbitrary point $[h,x]\in
\Riem(M)\times_{\Diff_{x_0}(M)} M$. The fibre at this point is of
course diffeomorphic to $M$. Let us now consider the tangent space to
this fibre. Suppose $(h,x), (h',x')\in \Riem(M)\times M $ both
represent the point $ [h,x]\in \Riem(M)\times_{\Diff_{x_0}(M)}
M$. Then the tangent spaces $T_{x}M$ and $T_{x'}M$ are isomorphically
related by the derivative map $\phi_{*}$ of some diffeomorphism
$\phi\in\Diff_{x_0}M$ which satisfies $ \phi(x)=x'$.  Thus, the
tangent space to $[h,x]$ can be thought of as the isomorphic
identification of all tangent spaces $T_{x'}M$ where $x'\in M$ lies in
the orbit of $x$ under the action of $\Diff_{x_0}M$.  Suppose now that
$[u], [v]$ denote tangent vectors to the fibre at $[x,h]$ represented
by tangent vectors $u,v\in T_{x}M$. We specify an inner product to the
tangent space to the fibre at $[h,x]$ by the following formula:
\begin{equation*}
\begin{split}
\langle [u],[v] \rangle_{[x,h]} = h_{x}(u,v)
\end{split}
\end{equation*}
where $h_x$ is the restriction of the Riemannian metric $h$ to the
tangent space $T_{x}M$.  It is an easy exercise to show that this is
well-defined and varies smoothly over $E(M)$; see \cite[p. 61]{WT}.
Notice that this does not give a Riemannian metric on $E(M)$ as we
only specify the inner product on fibres.

Given a map $f:X\rightarrow \M_{x_0}(M)$, this universal fibre metric
then pulls back to a continuous fibrewise family of Riemannian metrics
on $E_f$. More precisely, each fibre of the bundle $E_f\rightarrow X$,
already diffeomorphic to $M$, is now equipped with a Riemannian metric
which depends continuously on $X$. Clearly, varying the map $f$ by a
homotopy alters the fibrewise metric structure of the bundle.
Suppose, on the other hand, we begin with a fibrewise family of
metrics on an $M$-bundle over $X$.  Identifying fibres non-canonically
with a `standard' copy of $M$ and pulling back metrics leads to a
well-defined map $X\rightarrow \M_{x_0}(M)$.  Thus, we obtain a one to
one correspondence between maps $X\rightarrow \M_{x_0}(M)$ and
fibrewise families of metrics on $M$ which are parameterised by $X$.

Assuming $X$ is the sphere $S^{i}$, we are now brought back to the
homomorphism of homotopy groups
\begin{equation*}\label{homom}
  \iota_{*}:\pi_{i}(\M_{x_0}^{\Ric>0}(M), [g_0])\longrightarrow
  \pi_{i}( \M_{x_0}(M), [g_0]), 
\end{equation*}
induced by the inclusion $\iota:\M_{x_0}^{\Ric>0}(M)\hookrightarrow
\M_{x_0}(M)$. Let $f:S^{i}\rightarrow \M_{x_0}(M)$ represent an
element of $\pi_{i}(\M_{x_0}(M), [g_0])$. This element determines (and
is determined by) an $M$-bundle $E_f \to S^i$ as above, together with
a fibrewise family of metrics on $E_f$. Thus, it is possible to lift
this element of $\pi_{i}( \M_{x_0}(M), [g_0])$ to an element of
$\pi_{i}(\M_{x_0}^{\Ric>0}(M), [g_0])$, provided we can construct a
fibrewise family of positive Ricci curvature metrics on $E_f.$
\subsection{The work of Farrell and Hsiang} 
 At this stage we have established that lifting an element of
 $\pi_{i}( \M_{x_0}(M), [g_0])$ to $\pi_{i}(\M_{x_0}^{\Ric>0}(M),
 [g_0])$ involves the construction of a family of fiberwise Ricci
 positive Riemannian metrics on some bundle over $S^{i}.$ However, we
 have not yet discussed the particular elements in the homotopy groups
 of $ \M_{x_0}(M)$ which we plan to lift. It is here that we recall a
 result of Farrell and Hsiang, which identifies the rational homotopy
 groups of $\BDiff_{x_0}(S^n)$ in a stable range, using algebraic
 $K$-theory and Waldhausen $K$-theory computations; see
 \cite{FH}. Recalling that $\bM_{x_0}(S^{n})$ is homotopy equivalent
 to the classifying space $\BDiff_{x_0}(S^{n})$, the result of these
 computations can be stated as follows.
 \begin{theorem}\label{thm:farrel-hsiang}
   {\rm (Farrell and Hsiang, \cite{FH}.)}  For any $m\in \mathbb{N}$,
   there is an integer $N(m)$ such that for all odd $n>N(m)$ and
   $i\leq 4m$,
\begin{equation*}
\pi_i \bM_{x_0}(S^n)\otimes \Q =
\left\{
\begin{array}{cl}
\Q  & \mbox{if} \ n \ \mbox{odd and }i\equiv 0 \mod 4,
\\
0 & \mbox{otherwise}.
\end{array}
\right.
\end{equation*} 
\end{theorem}
\noindent Thus, for appropriate $i$, we now have lots of non-trivial groups
$\pi_{i}(\M_{x_0}(S^{n}), [g_0]) \otimes {\mathbb Q}$.  This also
explains the hypotheses of the main theorem.

This leaves the question of which $S^n$-bundles over $S^i$ ($i \equiv
0 \mod 4$) can represent the non-trivial elements of
$\pi_{i}(\M_{x_0}(S^{n}), [g_0]) \otimes {\mathbb Q}.$ It turns out
that these elements can be represented by so-called `Hatcher
bundles'. In section \ref{Hatch} we will provide a description of
these, based on the work of Hatcher and Goette (see \cite{GO}). Our
approach to the topological construction of Hatcher bundles is guided
by the geometric constructions we must subsequently perform, namely
the production of fibrewise Ricci positive metrics. These metric
issues will be addressed in section \ref{Construct}, and will involve
a generalized version of a powerful gluing theorem due to
Perelman. Perelman's theorem and our generalization of this is the
subject of section \ref{Perelman}.

This work was initiated while the third named author was visiting the
first, and he would like to thank the University of Oregon for their
hospitality. He would also like to thank Sebastian
Goette for a useful discussion about Hatcher
bundles, and Janice Love for her help with the Maple code used in
section \ref{Perelman}.

\section{Gluing manifolds and a theorem of Perelman}\label{sec:Perelman}
\subsection{The gluing construction}
The purpose of this section is to present a theorem of Perelman which
allows for the construction of Ricci positive metrics on a closed
manifold by gluing together certain Ricci positive metrics on
manifolds with boundary. This result is the principal geometric tool
used in achieving our goal of obtaining a fibrewise family of positive
Ricci curvature metrics on the total space of a Hatcher
bundle. Perelman's theorem is originally published
in \cite{P} and justified with a brief outline, omitting the
details. Our aim is to provide those details, and in so doing offer a
more comprehensive justification, currently lacking in the literature,
for a very useful result. In our experience this result is not widely
known, and we hope that by offering these details we will help provide
some of the intuition behind the construction, as well as
bringing it to a potentially wider audience. Moreover, these details are
important for establishing the family version of Perelman's theorem, 
which appears at the end of this section.

We begin with a brief review of the notion of gluing smooth manifolds,
something we make extensive use of throughout the paper. Consider a
pair of smooth $n$-dimensional manifolds, $M_{1}$ and $M_2$, each with
non-empty boundary. We further assume that $\partial M_1$ and
$\partial M_2$ are diffeomorphic via a diffeomorphism $\phi:\partial
M_1\rightarrow \partial M_2$. From this, we may form the adjunction
space, $W=M_{1}\cup_{\phi}M_{2}$, obtained as the quotient of
$M_{1}\sqcup M_2$ by identifying each $x\in \partial M_1$ with
$\phi(x)\in\partial M_2$. In particular, the quotient map embeds both
$M_1$ and $M_2$ into this space. For simplicity then, we identify
$M_1$ and $M_2$ with their images in $W$ and write $X=\partial
M_1=\partial M_2$.  Consider collar neighbourhoods $\partial M_1
\times (-\delta,0]$ and $\partial M_2 \times [0,\delta)$ about $X$ for
    some small $\delta>0,$ for example determined by the normal
    coordinate from $\partial M_1,$ $\partial M_2$ with respect to
    some choice of metrics on $M_1,$ $M_2$. Denote by $N$ the union of
    the images of these collar neighbourhoods in $W.$ We then have a
    homeomorphism between $X \times (-\delta,\delta)$ and $N$ given by
\begin{equation*}
(x,r)\mapsto
\left\{
\begin{array}{cl}
(m_1,r)  & \mbox{if} \ r\le 0
\\
(\phi(m_1),r) & \mbox{if} \ r\ge 0,
\end{array}
\right.
\end{equation*}
where $x \in X$ is the equivalence class $x=\{m_1,\phi(m_1)\}$ for
some $m_1 \in \partial M_1.$

We can now use this to give $N$ a differentiable
structure, by pulling back the standard differentable structure on
$X \times (-\delta,\delta)$ via the inverse homeomorphism. 
Finally, we extend this differentiable
structure over $M_1$ and $M_2$ to give a differentiable structure on
$W$. Although there are many choices involved in this construction,
leading to many possible differentiable structures, it is a well known
fact that the diffeomorphism type of the resulting smooth manifold $W$
is independent of these choices; see \cite[Ch. 8, Sec. 2]{Hirsch}.
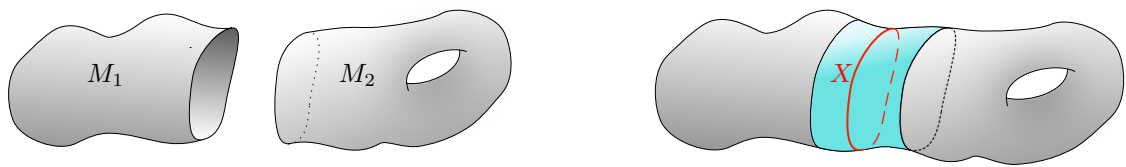
\begin{figure}[!htbp]
%\vspace{-1.0cm}
\begin{picture}(430,70)%
\includegraphics{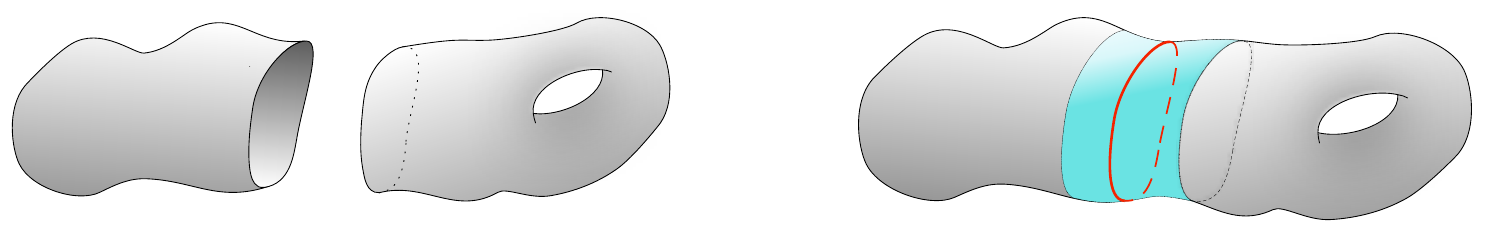}%
\end{picture}%
\setlength{\unitlength}{2500sp}%
\begingroup\makeatletter\ifx\SetFigFont\undefined%
\gdef\SetFigFont#1#2#3#4#5{%
  \reset@font\fontsize{#1}{#2pt}%
  \fontfamily{#3}\fontseries{#4}\fontshape{#5}%
  \selectfont}%
\fi\endgroup%
\begin{picture}(0,0)(430,70)
\put(-10000,1000){\makebox(0,0)[lb]{\smash{{\SetFigFont{10}{8}{\rmdefault}{\mddefault}{\updefault}{\color[rgb]{0,0,0}$M_1$}%
}}}}
\put(-7500,1000){\makebox(0,0)[lb]{\smash{{\SetFigFont{10}{8}{\rmdefault}{\mddefault}{\updefault}{\color[rgb]{0,0,0}$M_2$}%
}}}}
\put(-2630,1000){\makebox(0,0)[lb]{\smash{{\SetFigFont{10}{8}{\rmdefault}{\mddefault}{\updefault}{\color[rgb]{0,0,0}$\textcolor{red}{X}$}%
}}}}
\end{picture}%
\caption{The manifolds with boundary, $M_1$ and $M_2$ (left) along
  with the adjunction space $W$ and tubular neighborhood $N$ of
  $X\subset W$ (right)}
\label{gluing1}
\end{figure}   
We now consider such a gluing in the Riemannian setting, equipping
$M_1$ and $M_2$ with Riemannian metrics, $h_1$ and $h_2$. Let us
assume that the restrictions of these metrics to their respective
boundaries are isometric via $\phi$. More precisely, we assume
\begin{equation*}
h_1|_{\partial M_1}=\phi^{*}h_2|_{\partial M_2}.
\end{equation*}
This automatically leads to a well-defined $C^{0}$-metric
$h=h_1\cup_{\phi} h_2$, on $M_{1}\cup_{\phi}M_{2}.$ Notice that this
adjunction metric is smooth if and only if it is smooth in a collar
neighbourhood of $X \subset M_1 \cup_\phi M_2.$ In view of the of the
adjunction space discussion above, this will be the case if the metric
$h_1|_{\partial M_1 \times (-\delta,0]}$ glues smoothly with $(\phi
  \times id_{[0,\delta)})^*(h_2|_{\partial M_2 \times [0,\delta)}).$
\subsection{The theorem of Perelman}
The above construction gives the $C^0$-metric $h_1 \cup_\phi h_2$ on
$M_1\cup M_2$.  We will be interested in smoothing the metric $h_1
\cup_\phi h_2$ within positive Ricci curvature in the case where $h_1$
and $h_2$ individually have positive Ricci curvature. This is not
always possible. However, the following theorem of Perelman shows that
under certain additional assumptions involving the normal (i.e. the
principal) curvatures of $h_1$ and $h_2$ at the boundary, such a
smoothing can be performed.
\begin{theorem}\label{Perelman}
Let $(M_1, h_1)$ and $(M_2, h_2)$ be a pair of Riemannian manifolds
with positive Ricci curvature and $\phi:(\partial M_1, h_1|_{\partial
M_1})\rightarrow (\partial M_2, h_2|_{\partial M_2} )$, an isometry of
their boundaries. Suppose that the normal curvatures of $
h_1|_{\partial M_1}$ with respect to the outward normal are greater
than the negatives of corresponding normal curvatures of
$h_2|_{\partial M_2}$ with respect to its outward normal.  Then the
$C^{0}$-metric, $h=h_1\cup_{\phi} h_2$ on the smooth manifold
$M_{1}\cup_{\phi}M_{2}$ can be replaced by a $C^{\infty}$-metric with
positive Ricci curvature, agreeing with $h_1$ and $h_2$ outside a
neighbourhood of the glued boundaries.
\end{theorem}
\begin{proof}
As above, we will denote by $X$ the hypersurface of $M_1 \cup_\phi
M_2$ along which $M_1$ and $M_2$ are joined, and assume that the
normal parameter through the hypersurface $X$ gives rise to collar
neighbourhoods $\partial M_2 \times [0,\delta)$ in $M_2$ and $\partial
  M_1 \times (-\delta,0]$ in $M_1$, for some $\delta>0.$ Since we will
be working exclusively in a collar neighbourhood of $X$, for
convenience we can simply re-label the metric $(\phi \times
id_{[0,\delta)})^*(h_2|_{\partial M_2 \times [0,\delta)})$ by $h_2,$
    assume that $\partial M_1=\partial M_2$ and that $\phi$ is the
    identity map. Thus from now on we will write $h_1 \cup h_2$ for
    the $C^0$-metric in the theorem, and $M_1 \cup M_2$ for the
    manifold. We will introduce a parameter $t$, normal to $X$,
    running from $M_1$ to $M_2$, such that $t=0$ corresponds to
    $X$. Observe that $M_1 \cup M_2$ has a smooth topological
    structure (though not a smooth metric structure), and that with
    respect to this $t$ is a smooth parameter.

Choose a small parameter $\epsilon>0$. (We will say more about an
appropriate size for $\epsilon$ later.) Our next task is to write down
a new metric on $X \times [-\epsilon,\epsilon]$ which joins with $h_1$
for $t<\epsilon$ and $h_2$ for $t>\epsilon$ to give a $C^1$-metric on
$M_1 \cup M_2$. This new metric will take the form
$dt^2+g(t)$. Denoting by $h_i(t)$, where $i=1$ or $2$, the induced
metric on the hypersurface of constant distance $t$ from $X$, we will
choose $g(t)$ to be the following cubic expression in $t$:
\begin{equation}\label{eq:g(t)}
\begin{array}{lcl}
  g(t)&=& \displaystyle
  \frac{t+\epsilon}{2\epsilon}h_2(\epsilon)
  - \frac{t-\epsilon}{2\epsilon}h_1(-\epsilon)
   + \frac{(t-\epsilon)^2(t+\epsilon)}{4\epsilon^2}
 \left[h_1'(-\epsilon)-\frac{1}{2\epsilon}[h_2(\epsilon)-h_1(-\epsilon)]\right]
\\
  \\
  & & \displaystyle + \frac{(t+\epsilon)^2(t-\epsilon)}{4\epsilon^2}
 \left[h_2'(\epsilon)-\frac{1}{2\epsilon}[h_2(\epsilon)-h_1(-\epsilon)]\right].
\end{array}
\end{equation}
\begin{lemma}\label{lemma:cubic}
 Assume that the metrics $h_1$, $h_2$ satisfy the hypothesis from
 Theorem \ref{Perelman}. Then there exists $\epsilon>0$ such that
\begin{enumerate}
\item[(i)] with $g(t)$ as in {\rm (\ref{eq:g(t)})}, the metric
  $dt^2+g(t)$ is smooth if $t\neq 0$ and $C^1$ at $t=0$;

\item[(ii)] $\Ric_{dt^2+g(t)}>0$.
\end{enumerate}
\end{lemma}
\begin{proof}[{\it Proof of Lemma \ref{lemma:cubic}.}]
First, we find the $t$-derivative of this metric. A straightforward
calculation gives
\begin{equation*}
\begin{array}{lcl}
  g'(t) &=& \displaystyle
  \frac{1}{2\epsilon}
       [h_2(\epsilon)-h_1(-\epsilon)]
         +
  \frac{2(t^2-\epsilon^2)+(t-\epsilon)^2}{4\epsilon^2}
  \left[h_1'(-\epsilon)- \frac{1}{2\epsilon}
       [h_2(\epsilon)-h_1(-\epsilon)]\right]
  \\
  \\
  && \displaystyle +
  \frac{2(t^2-\epsilon^2)+(t+\epsilon)^2}{4\epsilon^2}
  \left[h_2'(\epsilon)- \frac{1}{2\epsilon}
    [h_2(\epsilon)-h_1(-\epsilon)]\right].
\end{array}
\end{equation*}
It is now an easy exercise to check that the metric $g$ forms a $C^1$
join with the $h_i$ at $t=\pm\epsilon$. (The metric $g$ is of course
smooth.)

Our next task is to investigate the curvature properties of
$dt^2+g(t)$.
We begin by assuming that the cubic expression for $g(t)$ above holds
in an open neighbourhood containing
$[-\epsilon,\epsilon]\times X$. Motivated by the fact
that curvature is a second derivative phenomenon, an easy calculation
shows that
\begin{equation*}
g''(t)=\frac{1}{4\epsilon^2}(6t-2\epsilon)\Bigl[h_1'(-\epsilon)- 
\frac{1}{2\epsilon}[h_2(\epsilon)-h_1(-\epsilon)]\Bigr] + \frac{1}{4\epsilon^2}
  (6t+2\epsilon)\Bigl[h_2'(\epsilon)- \frac{1}{2\epsilon}[h_2(\epsilon)-h_1(-\epsilon)]\Bigr]. 
\end{equation*}
We will investigate the limiting behaviour of $g''(\pm\epsilon)$ as
$\epsilon \to 0.$ At $t=\epsilon$ we have
\begin{equation}\label{eq:second-der}
  g''(\epsilon)=\frac{1}{\epsilon}\left[h_1'(-\epsilon)+2h'_2(\epsilon)-
    \frac{3}{2}\frac{h_2(\epsilon)-h_1(-\epsilon)}{\epsilon}\right].
\end{equation}
Consider the term $(h_2(\epsilon)-h_1(-\epsilon))/\epsilon$ in
(\ref{eq:second-der}).  In the limit $\epsilon \to 0$ we see by
l'H\^opital's rule that the value of the limit is $h_1'(0)+h_2'(0)$.
Clearly, the overall limit of the bracketed term in
(\ref{eq:second-der}) is
\begin{equation*}
h_1'(0)+2h_2'(0)-\frac{3}{2}(h_1'(0)+h_2'(0))=\frac{1}{2}(h_2'(0)-h_1'(0)).
\end{equation*}
A similar calculation shows that the corresponding term in
$\displaystyle \lim_{\epsilon \to 0} g''(-\epsilon)$ yields exactly
the same expression.

Next observe that $g''(t)$ has a linear dependence on $t$. Thus if
$g''(-\epsilon)$ and $g''(\epsilon)$ have the same sign, then this
sign persists for all $t \in [-\epsilon,\epsilon]$. We will show that
under the Perelman normal curvature assumption, the sign of both is
negative provided $\epsilon$ is chosen sufficiently small.\footnote{
  \ Notice what we have used so far. For the $C^1$ cubic expression we
  require no assumptions. In order to obtain the limiting formula for
  the second derivative we only need that the original metric on the
  union is continuous at $t=0$.}

Assume that $\epsilon$ is chosen so small that the topological product
structure in a neighbourhood of $X$ extends over $t \in
[-2\epsilon,2\epsilon]$. Let $u$ be a fixed vector tangent to $X$ at
some point $x_0$. As $u$ is independent of the metric, given the
product structure around $X$, we can consider the `same' vector for
nearby values of $t$, i.e. we obtain a local vector field $u$ along
the line ${(t,x_0)}$. Define the normal curvature function $k(u)$ to
be $k(u)=\langle \nabla_u \partial_t , u \rangle$, which is the normal
curvature for the vector $u$ of the hypersurface $t=\hbox{constant}$,
with $\partial/\partial t$ providing the normal direction.  Note that
we do not insist that $u$ is unit with respect to any metric. But now
observe that we can rearrange this definition to give
$k(u)=\frac{1}{2} \frac{\partial}{\partial t} \langle u,u \rangle$,
which is just $\frac{1}{2} g'(t)(u,u)$. Differentiating with respect
to $t$ we obtain $k'(u)=\frac{1}{2} g''(t)(u,u).$

The difference of the normal curvatures corresponding to $u$ across
$X$ can be viewed as
\begin{equation}\label{difference}
\frac{1}{2}\lim_{\epsilon \to 0}\Bigl(
g'(\epsilon)(u,u)-g'(-\epsilon)(u,u)\Bigr)=\frac{1}{2}(h_1'(0)-h_2'(0)).
\end{equation}
Denoting the normal curvatures at $\partial M_1$ and $\partial M_2$ with respect to the outward normals by $k_1,k_2$, it is straightforward to see that
\begin{equation*}
\begin{array}{lcl}
k_1(u)=&\lim_{t \to 0^-} k(u)&=\frac{1}{2}\lim_{t \to 0^-}g'(t)(u,u); \cr
k_2(u)=&-\lim_{t \to 0^+} k(u)&=-\frac{1}{2}\lim_{t \to 0^+}g'(t)(u,u). \cr
\end{array}
\end{equation*}
Thus the difference of normal curvatures \ref{difference} is equal to
$k_1(u)+k_2(u).$ Now the Perelman normal curvature assumption is that
$k_1(u)>-k_2(u),$ which means that the difference of normal curvatures
is positive, and hence $\displaystyle \lim_{\epsilon \to 0} \epsilon
g''(\pm \epsilon)=\frac{1}{2}(h_2'(0)-h_1'(0))<0.$ Therefore by
choosing $\epsilon$ sufficiently small, we can bound $g''(t)(u,u)$ as
\begin{equation}\label{eq:constant}
g''(t)(u,u) < -A \cdot |u|,
\end{equation}
where $-A$ is an arbitrarily large negative constant and the norm of
$u$ is taken with respect to $h_1$ or $h_2$ say, as these are common
on $X$.  In turn this means that $k'(u)$ can similarly be bounded
above.

The relevance of $k'(u)$ is that it can be re-written in terms of the
curvature tensor applied to $u$ and
$\partial_t:=\frac{\partial}{\partial t}$, and we can use the
arbitrarily negative feature of $k'(u)$ to produce an arbitrarily
large positive lower bound for $R(\partial_t,u,u,\partial_t)$. In
detail we have
\begin{equation*}
\begin{array}{lcl}
  k'(u) &=&
  \partial_t \langle \nabla_u \partial_t,u\rangle %\\
  %\\
  %& = &
  = \langle
  \nabla_{\partial_t}\nabla_u \partial_t,u \rangle+\langle \nabla_u
  \partial_t,\nabla_{\partial_t} u \rangle \\
  \\
  & =& \langle \nabla_t
  \nabla_u \partial_t,u \rangle+|S(u)|^2 
\end{array}
\end{equation*}
where $S(u)$ denotes the shape operator of the hypersurfaces given by
constant values of $t$, and where we have used the fact that
$\nabla_{\partial t} u=\nabla_u \partial t$ since $[\partial
t,u]\equiv 0$. On the other hand we have
\begin{equation*}
\begin{array}{lcl}
  R(\partial_t,u,u,\partial_t)& =&-R(\partial_t,u,\partial_t,u)
  \\ \\ &=& -\Bigl[\langle \nabla_{\partial_t} \nabla_u \partial_t,u
    \rangle-\langle \nabla_u\nabla_{\partial_t} \partial_t,u \rangle
    \Bigl] \\ \\ & =&-\langle \nabla_{\partial_t} \nabla_u
  \partial_t,u \rangle
\end{array}
\end{equation*}
as $\nabla_{\partial_t} \partial_t \equiv 0$. Thus we
conclude that
\begin{equation}\label{eq:shape-op}
R(\partial_t,u,u,\partial_t)=-k'(u)+|S(u)|^2.
\end{equation}
In particular, since $k'(u)= \frac{1}{2} g''(u,u)$, and $g''(t)(u,u)<
-A |u|$, as in (\ref{eq:constant}), for small enough $\epsilon>0$, we
can bound $R(\partial_t,u,u,\partial_t)$ below by any given positive
constant.

The observation is now that any Ricci curvature expression must
contain this large positive term, and we therefore get positive Ricci
curvature for the metric $dt^2+g(t)$, provided we show that other
curvature tensor expressions remain bounded. It is easily checked that
this boundedness reduces to showing that $\|R(u_i,u_j)u_k)\|$ is
bounded above by some constant independent of $\epsilon$ for all
vectors $u_i,u_j,u_k$ tangent to $X$ which are unit with respect to
say $h_1(0)=h_2(0).$

With the above curvature expression (\ref{eq:shape-op}) in mind,
consider the derivatives of the metric $g(t)$ in directions orthogonal
to $t$. The quantities
$h_1(-\epsilon),h_1'(-\epsilon),h_2(\epsilon),h_2'(\epsilon)$ and
their derivatives can clearly be bounded independent of $\epsilon$. We
also know that the other terms involving $\epsilon$:
\begin{equation*}
  \frac{t\pm\epsilon}{2\epsilon}  , \ \ \ \
  \frac{1}{2\epsilon}(h_2(\epsilon)-h_1(-\epsilon)) , \ \ \ \
  \frac{(t\pm
    \epsilon)^2(t \mp \epsilon)}{4\epsilon^2}
\end{equation*}
all remain bounded for $t \in [-\epsilon,\epsilon]$ as $\epsilon \to
0$. Therefore the derivatives of $g(t)$ orthogonal to $t$ must stay
bounded independent of $\epsilon$.

We also claim that the first derivative of $g(t)$ with respect to $t$
is bounded independently of $\epsilon$. This follows from the above
calculations involving $g''(\pm \epsilon)$, see
(\ref{eq:second-der}).  We showed that for $\epsilon$ sufficiently
small, the sign of $g''(t)(u,u)$ is negative, from which we see that
the values of $g'(t)(u,u)$ must lie between those at $t=\pm \epsilon,$
and hence are bounded independent of $\epsilon.$
We also notice that boundedness can then
also be deduced for $g'(t)(u,v)$ via the polarization formula for
inner products.

We conclude that the norm $ \|R(u_i,u_j)u_k\|$ is
bounded for all $t \in [-\epsilon,\epsilon]$ independent of $\epsilon$
provided the curvature $R(u_i,u_j)u_k$ does not depend on the second
derivative of the metric with respect to $t$. Without loss of
generality assume that $u_i,u_j,u_k$ are coordinate vector fields for
some coordinate system on $X$ extended to a coordinate system in a
neighbourhood of $X$ by the parameter $t$. The relevant expression for
the components of $R(u_i,u_j)u_k$ in terms of Christoffel symbols is
\begin{equation}\label{eq:curvature}
R^l_{ijk}=\partial_i\Gamma^l_{jk}-\partial_j\Gamma^l_{ik}+
\sum_{m=1}^n(\Gamma^m_{jk}\Gamma^l_{ im}-\Gamma^m_{ik}\Gamma^l_{jm}),
\end{equation}
where $l$ runs over all possible subscripts including $t$. Since the
Christoffel symbols in (\ref{eq:curvature}) have at most one
derivative with respect to $t$, we obtain the desired boundedness
property of $\|R(u_i,u_j)u_k)\|$.  Thus we fix small $\epsilon>0$ so
that the metric $g(t)$ has positive Ricci curvature and is smooth if
$t\neq0$, and $C^1$ if $t=0$. This proves Lemma
\ref{lemma:cubic}.\end{proof}
\vspace*{-2mm}

%\noindent
\noindent{\it Proof of Theorem \ref{Perelman} continued.}
Our next goal is to show how
to effect a $C^2$-smoothing of $g(t)$ in some
$\tau$-neighbourhoods of $t=\pm \epsilon,$ $\tau << \epsilon.$ As we
will make a very general construction, it will be convenient to
locally re-parameterize, and work in an interval $[-\tau,\tau].$

We start with an arbitrary $C^1$-function $f(t)$, which is smooth away
from $t=0$, and is defined in some open set of the real line
containing $[-\tau,\tau]$. Given such function $f(t)$, we will replace
$f(t)$ for $t \in [-\tau,\tau]$ by a quintic polynomial $p(t)$ which
will agree to second order with $f(t)$ at $t=\pm \tau.$ By applying
this idea to Riemannian metrics, we can create the desired
$C^2$-metric by quintic interpolation, in exactly the same way that we
created a $C^1$-metric using a cubic interpolation.

Let $p(t)=\sum_{n=0}^5 c_n t^n,$ and suppose that $f(\tau)=a_0,$
$f(-\tau)=b_0,$ $f'(\tau)=a_1,$ $f'(-\tau)=b_1,$ $f''(\tau)=a_2,$
$f''(-\tau)=b_2.$ Assuming that $p^{(i)}(\pm\tau)=f^{(i)}(\pm\tau)$
for $i=0,1,2$ yields a $(5 \times 5)$-linear system with the $c_n$ as
the unknowns and the $a_i,b_i$ as coefficients. Solving this system
(using Maple) shows that the polynomial $p(t)$ is uniquely determined
by the above requirements, and is equal to
\begin{equation}\label{eq:polyn}
\begin{split}
  p(t)=&\ \ \ \ \frac{\tau^2(a_2-b_2)-3\tau(a_1+b_1)+
    3(a_0-b_0)}{16\tau^5}t^5-\frac{-\tau(a_2+b_2)+(a_1-b_1)}{16\tau^3}t^4 \\
  &-\frac{\tau^2(a_2-b_2)-5\tau(a_1+b_1)+5(a_0-b_0)}{8\tau^3}t^3+
  \frac{-\tau(a_2+b_2)+3(a_1-b_1)}{8\tau}t^2 \\
  &+\frac{\tau^2(a_2-b_2)-7\tau(a_1+b_1)+15(a_0-b_0)}{16\tau}t+
  \frac{\tau^2(a_2+b_2)-5\tau(a_1-b_1)}{16}+\frac{a_0+b_0}{2}.
\end{split}
\end{equation}
With an eye towards curvature considerations when this quintic
interpolation has been applied to Riemannian metrics, consider next
the effect on $p(t)$ (for $t \in [-\tau,\tau]$) of letting $\tau \to
0.$ As this limit is approached, the term involving $t^5$ in the above
expression for $p(t)$ approaches
$\frac{3}{16}(a_0-b_0)\frac{t^5}{\tau^5},$ the $t^3$ term approaches
$-\frac{5}{8}(a_0-b_0)\frac{t^3}{\tau^3}$, and the first order term in
$t$ contributes $\frac{15}{16}(a_0-b_0)\frac{t}{\tau}$.

Recall that $|t|\leq |\tau|$, hence the limits of the
  $t^5$, $t^3$, and $t$- terms are bounded by
\begin{equation*} 
\begin{array}{c}
  \frac{3}{16}|a_0-b_0|, \ \ \ \frac{5}{8}|a_0-b_0|, \ \ \ \mbox{and} \ \ \
  \frac{15}{16}|a_0-b_0|
\end{array}
\end{equation*}
respectively. Since $|t|\leq |\tau|$, the degree four and two terms in
$t$ contribute nothing in the limit, and the zeroth order term yields
$\frac{a_0+b_0}{2}$.  In our case we can say more, however. Clearly,
the coefficients $a_i$ and $b_i$ are functions of $\tau$, i.e.,
$a_i=a_i(\tau)$ and $b_i=b_i(\tau)$ for $i=0,1,2,$ and since $f(t)$ is
assumed $C^1$ at $t=0$ we see that $\lim_{\tau \to 0}
a_j(\tau)=\lim_{\tau \to 0} b_j(\tau)=f^{(j)}(0)$ for $j=0,1.$ Thus we
conclude that for $\tau$ sufficiently small, the polynomial $p(t)$ can
$C^0$-approximate the constant function with value
$(a_0+b_0)/2=a_0=b_0=f(0)$ over the interval $[-\tau,\tau]$ to within
any desired degree.

Applying the same analysis to $p'(t)$ shows that for $\tau$
sufficiently small, $p'(t)$ $C^0$-approximates the constant function
with value $(a_1+b_1)/2=a_1=b_1=f'(0)$ over the interval
$[-\tau,\tau]$ to within any desired degree. In other words, by
choosing $\tau$ sufficiently small, $p(t)$ will $C^1$-approximate
$f(t)$ over $[-\tau,\tau]$ to within any desired accuracy.

Finally, we must consider the behaviour of $p''(t).$ Analogous
arguments to the above show that for $\tau$ sufficiently small,
$p''(t)$ can be $C^0$-approximated over $[-\tau,\tau]$ to within any
desired degree of accuracy by the cubic
\begin{equation}
\begin{array}{c}
  \frac{5}{4}(a_2-b_2)\frac{t^3}{\tau^3}-
  \frac{3}{4}(a_2-b_2)\frac{t}{\tau}+\frac{1}{2}(a_2+b_2) \ .
\end{array}
\end{equation}
To understand the behaviour of this cubic, it clearly suffices to
examine the function $q(t)=\frac{5}{\tau^3}t^3-\frac{3}{\tau}t$ over
$t \in [-\tau,\tau].$ An elementary calculation shows that maximum and
minimum values taken by $q(t)$ over $ [-\tau,\tau]$ are $q(\tau)=2$
respectively $q(-\tau)=-2.$ This function is depicted in
Fig. \ref{C2interpolate} below.
% Graph of $q$
% Graph of $q$
\begin{figure}[!htbp]
\vspace{3.5cm}
\hspace{0cm}
\begin{picture}(0,0)
\includegraphics{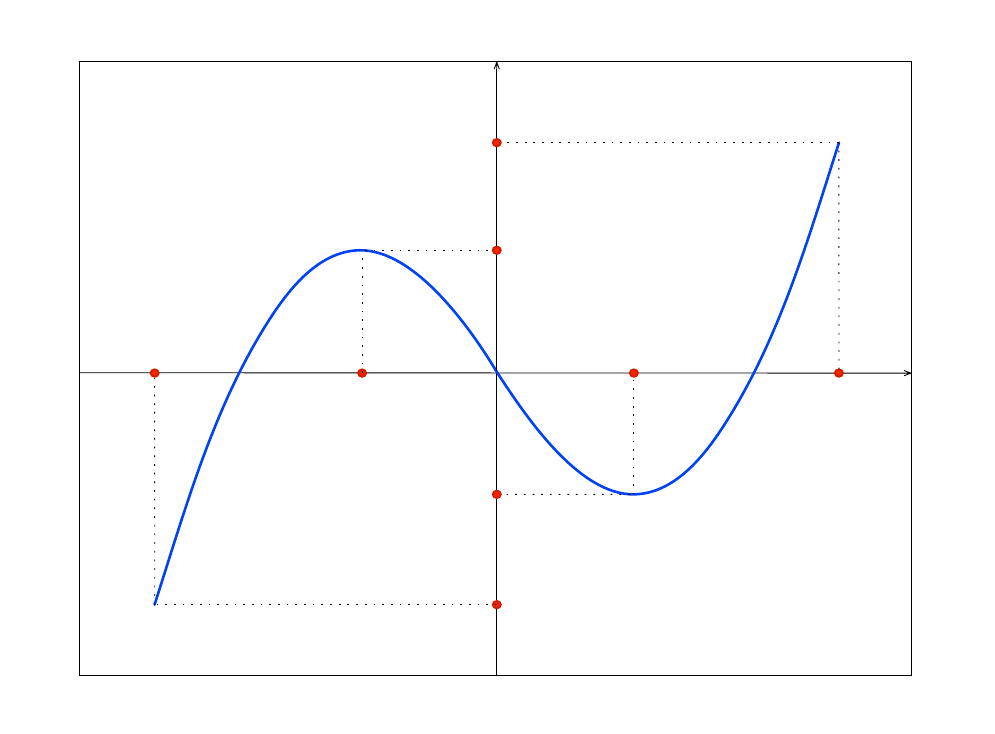}%
\end{picture}
\setlength{\unitlength}{3947sp}%\includegraphics{RelGLpic/
\begingroup\makeatletter\ifx\SetFigFont\undefined%
\gdef\SetFigFont#1#2#3#4#5{%
  \reset@font\fontsize{#1}{#2pt}%
  \fontfamily{#3}\fontseries{#4}\fontshape{#5}%
  \selectfont}%
\fi\endgroup%
\begin{picture}(5079,1559)(1902,-7227)
\put(2300,-5600){\makebox(0,0)[lb]{\smash{{\SetFigFont{10}{8}{\rmdefault}{\mddefault}{\updefault}{\color[rgb]{1,0,0}$-\tau$}%
}}}}
\put(3300,-5650){\makebox(0,0)[lb]{\smash{{\SetFigFont{10}{8}{\rmdefault}{\mddefault}{\updefault}{\color[rgb]{1,0,0}$-\frac{\tau}{\sqrt 5}$}%
}}}}

\put(4075,-4430){\makebox(0,0)[lb]{\smash{{\SetFigFont{10}{8}{\rmdefault}{\mddefault}{\updefault}{\color[rgb]{1,0,0}$2$}%
}}}}
\put(4250,-4970){\makebox(0,0)[lb]{\smash{{\SetFigFont{10}{8}{\rmdefault}{\mddefault}{\updefault}{\color[rgb]{1,0,0}$\frac{2}{\sqrt{5}}$}%
}}}}

\put(3800,-6120){\makebox(0,0)[lb]{\smash{{\SetFigFont{10}{8}{\rmdefault}{\mddefault}{\updefault}{\color[rgb]{1,0,0}$-\frac{2}{\sqrt{5}}$}%
}}}}
\put(4250,-6680){\makebox(0,0)[lb]{\smash{{\SetFigFont{10}{8}{\rmdefault}{\mddefault}{\updefault}{\color[rgb]{1,0,0}$-2$}%
}}}}

\put(5730,-5600){\makebox(0,0)[lb]{\smash{{\SetFigFont{10}{8}{\rmdefault}{\mddefault}{\updefault}{\color[rgb]{1,0,0}$\tau$}%
}}}}
\put(4650,-5650){\makebox(0,0)[lb]{\smash{{\SetFigFont{10}{8}{\rmdefault}{\mddefault}{\updefault}{\color[rgb]{1,0,0}$\frac{\tau}{\sqrt 5}$}%
}}}}

\end{picture}%
\caption{The graph of the function $q(t)$}
\label{C2interpolate}
\end{figure}
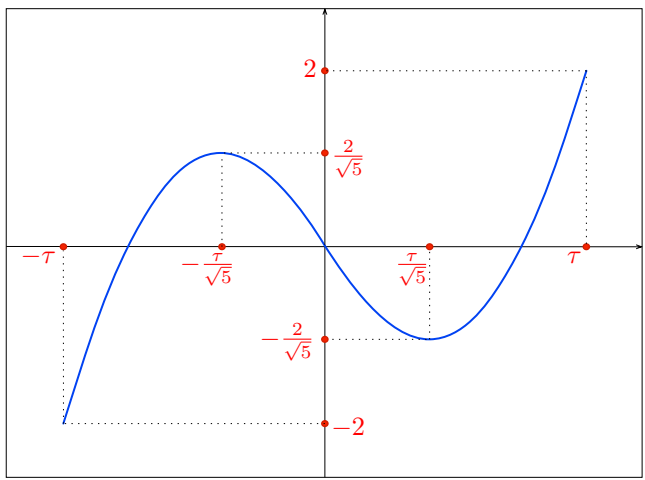   

\noindent
Thus the values of $q(t)$ for all other $t$ in this interval lie
between the endpoint values. It follows immediately that the same is
true for $p''(t)$ (with respect to its endpoint values).
\begin{lemma}\label{lem:c-2-approximation}
Let $g(t)$ be the $C^1$-metric given by (\ref{eq:g(t)}), where
$\epsilon>0$ is chosen so that $\Ric_{dt^2+g(t)}>0$. Then there exists
small enough $\tau << \epsilon$ and a $C^2$-metric $dt^2+\tilde g(t)$
such that
\begin{enumerate}
\item[(i)] $\Ric_{dt^2+\tilde g(t)}>0$;
\item[(ii)] $\tilde g(t)= g(t)$ if $|t|\geq \tau$;
\item[(iii)] the metrics $dt^2+\tilde g(t)$ and $dt^2+g(t)$ are
  arbitrarily $C^1$-close on $X\times [-\epsilon,\epsilon]$.
\end{enumerate}
\end{lemma}
\begin{proof} [{\it Proof of Lemma \ref{lem:c-2-approximation}.}]
  Indeed, we choose a coordinate system on $X\times [-\epsilon,
    \epsilon]$, so that the metric $g(t)$ is given by its components
  $g_{ij}(t)$, which are smooth when $t\neq 0$ and $C^1$-functions
  when $t=0$.  Then we use the polynomial (\ref{eq:polyn}) for each
  function $g_{ij}(t)$ to obtain functions $\tilde g_{ij}(t)$.
  Clearly by choosing $\tau$ sufficiently small we can bound the
  variation in the metric components $\tilde g_{ij}(t)$ and their
  first derivatives by an arbitrarily small constant, whereas the
  second derivatives vary between their values at the endpoints. As
  curvature is a $C^2$ phenomenon which depends linearly on the second
  derivatives of the metric, we see that any open convex curvature
  condition satisfied by both `halves' of the $C^1$ metric
  (i.e. either side of $t=\pm\epsilon$) will continue to be satisfied
  by the resulting $C^2$ metric, as $\tau$ is chosen to be
  sufficiently small. Since the positivity of the Ricci curvature is
  an open and convex condition we deduce that the our $C^2$ metric
  $dt^2+\tilde g(t)$ will have positive Ricci curvature if $\tau$ is
  sufficiently small.
\end{proof}
\noindent{\it Proof of Theorem \ref{Perelman} continued.} It remains
to smooth the metric from $C^2$ to $C^\infty.$ By general smoothing
theory for functions, we know that the set of $C^2$ functions on a
smooth manifold is dense in the space of $C^\infty$ functions (see for
Theorem 2.6 of \cite{Hirsch}). Thus we can make a $C^2$-arbitrarily
small adjustment to our $C^2$ metric to render it smooth, and in so
doing ensure the positivity of the Ricci curvature is preserved. This
proves Theorem \ref{Perelman}.
\end{proof}
Theorem \ref{Perelman} immediately gives us the following corollary,
which will play a key role in section 4.
\begin{cor}\label{pcor}
  The conclusion of Theorem \ref{Perelman} holds if the normal
  curvatures at both boundaries (with respect to the outward normals)
  are all positive.
\end{cor}
\subsection{A family version of Perelman's Theorem}
We will also need a family version of Theorem \ref{Perelman}, which
allows us to perform simulataneous Ricci positive smoothings on the
fibres of a bundle.
\begin{theorem}\label{family}
Let $\pi_i:E_i \to B$, $i=1,2$ be smooth compact fibre bundles with
fibre $M_i,$ where $\partial M_i \neq \emptyset.$ Suppose that each of
these bundles is equipped with a smoothly varying family of fibrewise
Ricci positive metrics $\{h_i(b)\}_{b \in B},$ and that with respect
to these metrics, there is a smoothly varying family of fibrewise
isometries $\phi:=\{\phi_b\}_{b \in B}$ for the boundary bundles
$\partial \pi_i:\partial E_i \to B$ (with fibre $\partial M_i),$ that
is, $\phi_b: \partial \pi_1^{-1}(b) \cong \partial \pi_2^{-1}(b)$ for
each $b \in B.$ Then provided the normal curvatures of
$h_1(b)|_{\partial \pi_1^{-1}(b)}$ with respect to the outward normal
are always greater than the negatives of corresponding normal
curvatures of $h_2(b)|_{\partial \pi_2^{-1}(b)}$ with respect to its
outward normal, the fibrewise $C^{0}$-metric
$h:=\{h_1(b)\cup_{\phi(b)} h_2(b)\}_{b \in B}$ on $E_1\cup_{\phi}E_2$
can be smoothed within fibrewise positive Ricci curvature in such a
way that the resulting metric agrees with the original ouside a
neighbourhood of the glued boundaries.
\end{theorem}
\begin{proof}
The key observation is that in the proof of Theorem \ref{Perelman},
the $C^2$-smoothing constructed there only depends on the metrics
together with two small positive parameters $\epsilon$ and $\tau.$ Now
suppose we have a smooth variation of the metrics on $M_1$ and $M_2,$
which nevertheless always satisfies the requirements of Theorem
\ref{Perelman}. It is clear that $\epsilon$, the first chosen
parameter in the construction which together with the given metrics
determines the $C^1$ smoothing, can be chosen to vary continuously
with the metric. Similarly the second parameter, $\tau,$ needed to
construct the $C^2$ smoothing, can be chosen to vary continuously with
the metrics and $\epsilon.$

In the situation of the current Theorem, it follows from the above
observations and the compactness of $B$ that we can make uniform
choices for $\epsilon$ and $\tau$ which will work for all fibres in
our bundles $E_1$ and $E_2.$ Having made these choices, the $C^2$
metric smoothing performed after gluing each pair of fibres is then
completely determined by the metrics on these fibres. Moreover, since
this is a smoothing by polynomials, it follows trivially that the
resulting metrics will vary smoothly from fibre to fibre.

Finally, the same argument as employed at the end of the proof of
Theorem \ref{Perelman} shows that our fibrewise Ricci positive metric
on $E_1 \cup_\phi E_2$ can be smoothed to class $C^\infty$ within
fibrewise positive Ricci curvature. (We could always extend our
fibrewise $C^2$ metric to a global $C^2$ metric for which the
intrinsic fibre metrics have positive Ricci curvature. This can then
be globally smoothed by a $C^2$-arbitrarily small deformation,
preserving the intrinsic positive Ricci curvature on the fibres, then
restricting to the fibres yields the desired smooth fibrewise metric.)
\end{proof}

\section{Hatcher bundles}\label{Hatch}
\subsection{Goette's Theorem}\label{subsec:goette}
The aim of this section is to review the construction and properties
of certain smooth $S^{n}$-bundles over $S^{i}$ known as ``Hatcher
bundles''. In short, a Hatcher bundle $E_{\lambda}\to S^i$ is a smooth
$S^n$-bundle determined by an element $\lambda \in \ker J$, where $J:
\pi_{i-1} O(p)\longrightarrow \pi_{i-1+p}S^{p}$ is the
$J$-homomorphism, and where $0<p<n$. A Hatcher bundle $E_{\lambda}\to
S^i$ has structure group $\Diff_{x_0}(S^n)$ and thus is classified by
some map $f_{\lambda}: S^i \to \BDiff_{x_0}(S^n)$. We then say that a
Hatcher bundle $E_{\lambda}\to S^i$ represents the element
$[f_{\lambda}]\in \pi_i \BDiff_{x_0}(S^n)$; below we identify
$\BDiff_{x_0}(S^n)=\bM_{x_0}(S^n)$.

In the introduction we stated a theorem of Farrell and Hsiang (Theorem \ref{thm:farrel-hsiang}) 
concerning the rational homotopy groups
$\pi_i \bM_{x_0}(S^n)\otimes \Q$. In fact, each element of
those groups may be represented by a Hatcher bundle.
\begin{theorem}\label{thm:goette}{\rm (See \cite[Section 5]{GO})}
Suppose that $n$ and $k$ satisfy the hypotheses of Theorem
\ref{thm:farrel-hsiang}, so $\pi_{4k} \bM_{x_0}(S^n)\otimes \Q\cong\Q$
for such $k$ and $n$.  Then for each element $[f]\in\pi_{4k}
\bM_{x_0}(S^n)\otimes \Q$, there is an integer $p$ with $0<p<n,$ and
an element $\lambda\in\ker J$, where $J: \pi_{i-1} O(p)\longrightarrow
\pi_{i-1+p}S^{p}$ is the $J$-homomorphism, such that the Hatcher
bundle $E_{\lambda}$ represents the element $[f]$.
\end{theorem}
We recall that the key feature of these bundles is that they are
\emph{exotic smooth $S^n$-bundles}, in the sense that each one is
homeomorphic to, but not diffeomorphic to, the trivial bundle
$S^{i}\times S^{n}\rightarrow S^{i}$. We will develop Goette's
construction so as to provide the appropriate setting for our
geometric arguments in the next section.  An in-depth description of
these bundles and their properties is given in \cite{GO},
and we refer the reader to this paper for further details.
\subsection{Preliminary constructions}
Throughout this section, we will assume that $n$ is odd and is
sufficiently large for all of our purposes. The groups $\pi_i
\bM_{x_0}(S^n)\otimes \Q$ are trivial unless $i=4k$ for appropriate
$k$, and so we will consider only bundles which have base manifold
$S^{4k}$ and fibre $S^{n}$.

Let us begin with the trivial bundle $S^{4k}\times S^{n}\rightarrow
S^{4k}$. By decomposing the fibre sphere $S^{n}$ into a pair of
northern and southern hemispherical discs, $D_{+}^{n}$ and
$D_{-}^{n}$, we can decompose the entire bundle into a pair of disc
bundles, $S^{4k}\times D_{+}^{n}\rightarrow S^{4k}$ and $S^{4k}\times
D_{-}^{n}\rightarrow S^{4k}$, glued together in the obvious way. Thus
the trivial bundle $S^{4k}\times S^{n}\rightarrow S^{4k}$ can be
regarded as the double of the trivial disc bundle $S^{4k}\times
D^{n}\rightarrow S^{4k}$. We will always assume that discs are closed
unless otherwise stated.

To construct a Hatcher bundle, we will make certain adjustments to the
trivial $D^{n}$-bundle over $S^{4k}$ to obtain a smooth bundle which
is homeomorphic to, but not diffeomorphic to the trivial disc
bundle. We will then form the double of this exotic $D^{n}$-bundle to
obtain the desired exotic $S^{n}$-bundle over $S^{4k},$ which will
represent a non-trivial element of $\pi_{4k} \bM_{x_0}(S^n)\otimes
\Q$.

We begin with the trivial disc bundle $S^{4k}\times D^{n}\rightarrow
S^{4k}$. The fibre $D^{n}$ decomposes as
\begin{equation*}
\begin{split}
D^{n}&=D^{p+1}\times D^{q}\\
&=\left(D^{p+1}(\rho)\times D^{q}\right)\cup
\left(S^{p}\times [\rho, 1]) \times D^{q}\right),
\end{split}
\end{equation*}
where $p+q+1=n$, $\rho\in (0,1),$ and $D^{p+1}(\rho)\times D^{q}$ is a
smaller version of the original disc product (with the $D^{p+1}$
factor having radius $\rho$) surrounded by an annular region
$(S^{p}\times [\rho, 1]) \times D^{q}$. The integers $p$ and $q$ may
be assumed to be positive; in fact at various stages in the
construction, it is necessary to allow both $p$ and $q$ to be
large. It will be convenient for later considerations to reorder the
factors and write the annular region as $S^{p} \times D^{q}\times
[\rho, 1]$. Henceforth we will denote this by $A$ and the remaining
piece, $D^{p+1}(\rho)\times D^{q}$, by $P$.\footnote{ \ This smaller
  product of discs resembles an ice-hockey puck, hence the notation.}
Thus, as illustrated in Fig. \ref{Hatcherfibre0} below, we have
\begin{equation*}
D^{n}=A\cup P.
\end{equation*}
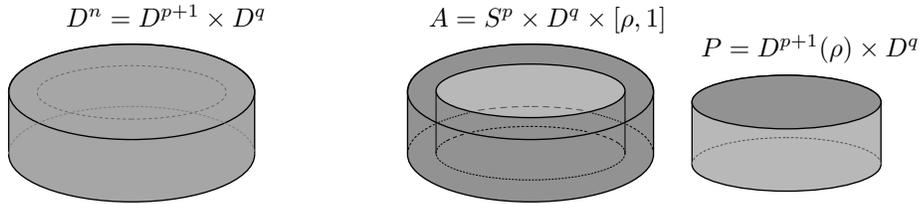
\begin{figure}[!htbp]
\vspace{0cm}
\begin{picture}(0,0)%
\includegraphics{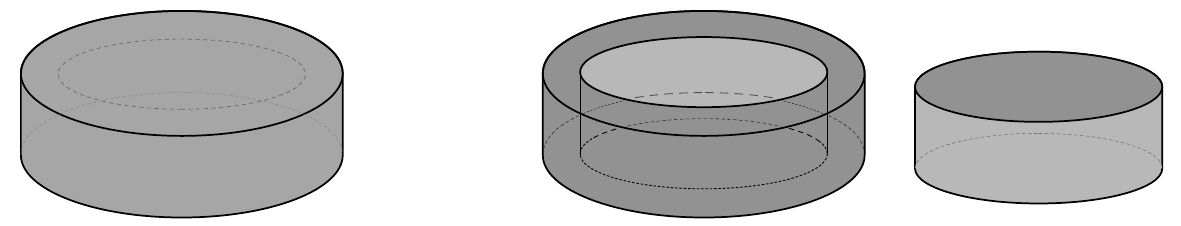}%
\end{picture}%
\setlength{\unitlength}{2500sp}%
\begingroup\makeatletter\ifx\SetFigFont\undefined%
\gdef\SetFigFont#1#2#3#4#5{%
  \reset@font\fontsize{#1}{#2pt}%
  \fontfamily{#3}\fontseries{#4}\fontshape{#5}%
  \selectfont}%
\fi\endgroup%
\begin{picture}(9000,2200)(0,0)
\put(700,1800){\makebox(0,0)[lb]{\smash{{\SetFigFont{10}{8}{\rmdefault}{\mddefault}{\updefault}{\color[rgb]{0,0,0}$D^{n}=D^{p+1}\times D^{q}$}%
}}}}
\put(4300,1800){\makebox(0,0)[lb]{\smash{{\SetFigFont{10}{8}{\rmdefault}{\mddefault}{\updefault}{\color[rgb]{0,0,0}$A=S^{p}\times D^{q}\times [\rho, 1]$}%
}}}}
\put(7000,1500){\makebox(0,0)[lb]{\smash{{\SetFigFont{10}{8}{\rmdefault}{\mddefault}{\updefault}{\color[rgb]{0,0,0}$P=D^{p+1}(\rho)\times D^{q}$}%
}}}}
\end{picture}%
\caption{The decomposition of the fibre disc $D^{n}$ into $A$ and $P$}
\label{Hatcherfibre0}
\end{figure}   
\noindent
The regions, $A$ and $P$, share a common piece of boundary,
$S^{p}\times D^{q}\times\{\rho\},$ and are glued together via the
identity map on $S^{p}\times D^{q}$.  The trivial bundle $S^{4k}\times
D^{n}\rightarrow S^{4k}$, therefore, can be thought of as a union of
sub-bundles $S^{4k}\times A\rightarrow S^{4k}$ and $S^{4k}\times
P\rightarrow S^{4k}$, glued together in the obvious way.

From now on we will write $A_{y}=\{y\}\times A$ and $P_{y}=\{y\}\times
P$, to denote the fibres at $y\in S^{4k}$ of the respective
sub-bundles $S^{4k}\times A\rightarrow S^{4k}$ and $S^{4k}\times
P\rightarrow S^{4k}$. Below we will specify a smooth diffeomorphism
\begin{equation}\label{eq:lambda}
  \Lambda_y: S^{p}\times D^{q}\rightarrow S^{p}\times D^{q}
\end{equation}
  over each $y\in S^{4k},$ where the domain is $(\p D^{p+1}) \times
  D^q \subset \p P_y $ and the target space is the
    product $S^p \times D^q \times \{\rho\} \subset \p A_y $. The idea
    will be to replace the identity map on $S^p\times D^q,$ which
    glues $P_{y}$ to $A_{y}$ to form $D^{p+1} \times D^q,$ with the
    map $\Lambda_{y}$.

Before we can begin the construction we will need a further
decomposition: that of the base manifold $S^{4k}$ into northern and
southern hemispherical discs
\begin{equation*}
  S^{4k}=D_{+}^{4k}\cup D_{-}^{4k}.
\end{equation*}
Over the disc $D_{-}^{4k}$ we take the trivial bundle $D_{-}^{4k}
\times D^n \to D_{-}^{4k},$ that is, we define the map $\Lambda_y$ to
be the identity map on $S^p \times D^q$ for all $y \in D_{-}^{4k}.$ We
therefore need to specify the maps $\Lambda_y$ for $y \in D_{+}^{4k}$
in order to describe the bundle over $D_{+}^{4k},$ and finally we need
to show how to glue the two bundles together over $S^{4k-1}=\partial
D_{-}^{4k}=\partial D_{+}^{4k}.$

Over the disc $D_+^{4k}$ we will actually work with a slightly
different, though topologically equivalent annulus ${\mathcal A}_y,$
which we will define below. We will also work with diffeomorphisms
$\Lambda_y$ as above, however we must adjust the target space to lie
in the boundary of ${\mathcal A}_y.$ Gluing the $P_y$ to ${\mathcal
  A}_y$ creates a fibre bundle over $D_+^{4k}$ with fibres $P_y
\cup_{\Lambda_y} {\mathcal A}_y.$

In order to make these constructions, let us first suppose we have a
collection of embeddings
\begin{equation*}
\bar{\lambda}_y:S^{p}\times D^{q}\longrightarrow S^{p}\times D^{q}
\end{equation*} for $y\in D_{+}^{4k}$ which
vary smoothly with $y$.
\begin{figure}[!htbp]
\vspace{0cm}
\begin{picture}(0,0)%
\includegraphics{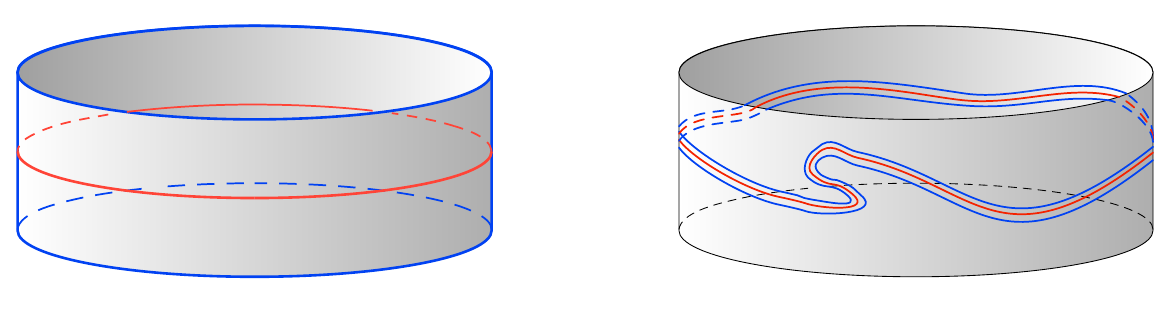}%
\end{picture}%
\setlength{\unitlength}{2500sp}%
\begingroup\makeatletter\ifx\SetFigFont\undefined%
\gdef\SetFigFont#1#2#3#4#5{%
  \reset@font\fontsize{#1}{#2pt}%
  \fontfamily{#3}\fontseries{#4}\fontshape{#5}%
  \selectfont}%
\fi\endgroup%
\begin{picture}(9000,2500)(0,0)
\end{picture}%
\caption{An embedding $\bar{\lambda}_y$ from $S^{p}\times D^{q}$ (left) into $S^{p}\times D^{q}$ (right)}
\label{HatcherEmb}
\end{figure}   
For any given $y\in D_{+}^{4k}$, the image of the embedding
$\bar{\lambda}_{y}$, which we denote $\im \bar{\lambda}_{y}$, is
schematically depicted in Fig. \ref{HatcherEmb}. We will define the
annulus ${\mathcal A}_y$ by
\begin{equation*}
  {\mathcal A}_y = \{y\}\times \im \bar{\lambda}_{y} \times [\rho,1]
  \subset \{y\}\times S^{p}\times D^{q}\times [\rho, 1].
\end{equation*}
We will furthermore define the diffeomorphism $\Lambda_y$ to be simply
the map $\bar{\lambda}_y$ with target space $\im \bar{\lambda}_y.$
Thus we can glue $P_y$ to ${\mathcal A}_y$ using $\Lambda_y$, by
identifying the  points
\begin{equation*}
(y,x)\in \{y\}\times (S^{p}\times D^{q})
\subset \partial P_y \ \ \ \mbox{and} \ \ \ (y, \bar{\lambda}_{y}(x),
\rho)\in\{y\}\times \im \bar{\lambda}_{y}\times \{\rho\}\subset
\partial {\mathcal A}_y,
\end{equation*} 
where $x\in S^{p}\times D^{q}$.  The spaces $P_y$ and ${\mathcal A}_y$
(as a subset of $\{y\}\times S^{p}\times D^{q}\times [\rho, 1]$) are
depicted in Fig. \ref{HatcherEmb2} below. Applying this gluing
fibrewise for all $y\in D_{+}^{4k}$ gives rise to the desired bundle
over $D_{+}^{4k}.$ For convenience we will denote the bundles over
$D^{4k}_{\pm}$ by ${\mathcal E}_{\pm}\rightarrow D_{\pm}^{4k}$.
\begin{figure}[!htbp]
\vspace{-1.0cm}
\begin{picture}(425,150)%
\includegraphics{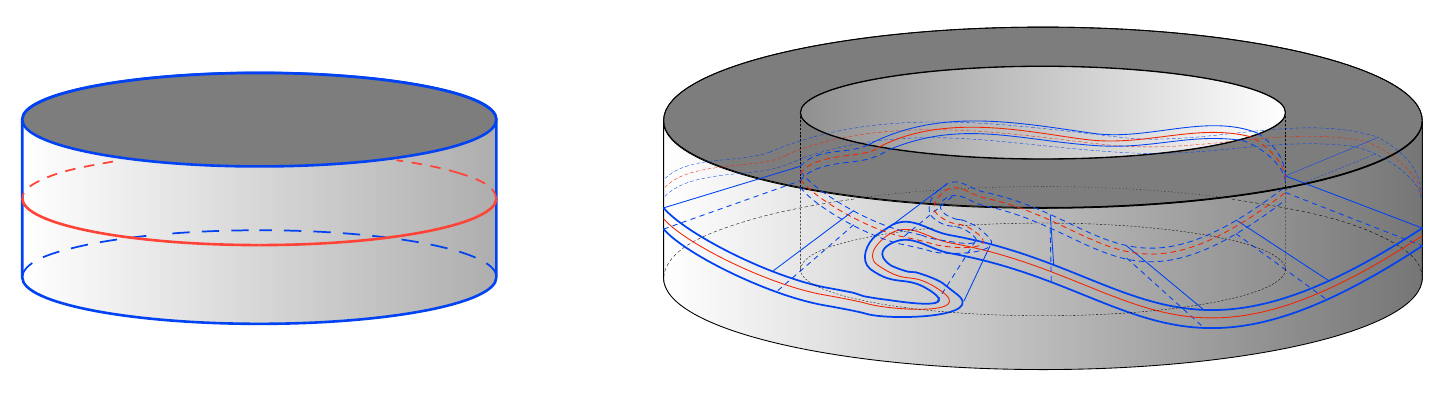}%
\end{picture}%
\setlength{\unitlength}{2500sp}%
\begingroup\makeatletter\ifx\SetFigFont\undefined%
\gdef\SetFigFont#1#2#3#4#5{%
  \reset@font\fontsize{#1}{#2pt}%
  \fontfamily{#3}\fontseries{#4}\fontshape{#5}%
  \selectfont}%
\fi\endgroup%
%\begin{picture}(9000,2500)(0,0)
%\put(700,1800){\makebox(0,0)[lb]{\smash{{\SetFigFont{10}{8}{\rmdefault}{\mddefault}{\updefault}{\color[rgb]{0,0,0}$D^{p+1}\times D^{q}$}%
%}}}}
%\put(4300,1800){\makebox(0,0)[lb]{\smash{{\SetFigFont{10}{8}{\rmdefault}{\mddefault}{\updefault}{\color[rgb]{0,0,0}$S^{p}\times [\rho, 1]\times D^{q}$}%
%}}}}
%\put(7000,1500){\makebox(0,0)[lb]{\smash{{\SetFigFont{10}{8}{\rmdefault}{\mddefault}{\updefault}{\color[rgb]{0,0,0}$D^{p+1}(\rho)\times D^{q}$}%
%}}}}
%\end{picture}%
\caption{The spaces $P_y$ (left) and ${\mathcal A}_y$, as a subset of $S^{p}\times D^{q}\times [\rho, 1]$ (right)}
\label{HatcherEmb2}
\end{figure}   
We claim that for each $y \in D_+^{4k}$ we have a canonical
diffeomorphism
\begin{equation*}
  D^{p+1} \times D^q = P_y \cup_{Id} A_y \cong P_y
  \cup_{\Lambda_y} {\mathcal A}_y.
\end{equation*}
To this end, define a map $\phi_y:A_y \to {\mathcal A}_y$ by setting
$\phi_y=\Lambda_y \times Id_{[\rho,1]},$ where we are viewing
$A_y=(S^p \times D^q) \times [\rho,1]$ and ${\mathcal A}_y=\im
\bar{\lambda}_y \times [\rho,1].$ Using this map we can define a
further map $\Phi_y:P_y \cup_{Id} A_y \to P_y \cup_{\Lambda_y}
{\mathcal A}_y$ by
\begin{equation}
  \Phi_y(z)=\left\{\begin{array}{ll}
  z & z \in P_y \\
  \phi_y(z) & z \in A_y. 
\end{array}
\right.
\end{equation}
It is clear that $\Phi_y$ is a homeomorphism for each $y$. Following
the discussion on adjunction spaces at the start of section
\ref{Perelman} (or see \cite[Chapter 8, section 2]{Hirsch}),
%chapter 8 section 2),
we see that with respect to the canonical differentiable
structure on $P_y \cup_{\Lambda_y} {\mathcal A}_y,$ $\Phi_y$ is
actually a diffeomorphism. Thus the exotic structure which Hatcher
bundles display does not occur at the level of individual fibres: it
is a global bundle phenomenon.

\subsection{Recollection of the $J$-homomorphism}
Here we follow \cite{GO} to construct the family of embeddings
$\bar{\lambda}_{y}:S^{p}\times D^{q+1}\rightarrow S^{p}\times
D^{q+1}$, for $y\in D_{+}^{4k}$, as above.  This requires us to
consider the $J$-homomorphism, which is a map
\begin{equation*}
J : \pi_{4k-1}O(p)\longrightarrow \pi_{4k-1+p}S^{p},
\end{equation*}
where $p$ is sufficiently large.\footnote{\ Here we assume that a base
  point in $S^p$ is the north pole.}  We can think of the
$J$-homomorphism as follows. Consider a map $\lambda:S^{4k-1} \to
O(p)$ determined by a choice of element $[\lambda] \in
\pi_{4k-1}O(p);$ by the Whitney approximation theorem, without loss of
generality we can assume that $\lambda$ is smooth. This then
determines a map $S^{4k-1}\times \mathbb{R}^{p} \to \mathbb{R}^{p},$
by sending a point $(y,z)$ to the orthogonal transformation
$\lambda(y)\in O(p)$ applied to $z\in \mathbb{R}^{p}$.  Since $S^{p}=
\mathbb{R}^{p}\cup \{x_0\}$, it is convenient to identify $\lambda$
with a map $ \lambda : S^{4k-1}\rightarrow \mathcal{C}(S^{p}, S^{p})$,
where $\mathcal{C}(X, Y)$ is a space (with compact-open topology) of
continuous maps $X\to Y $ preserving base-points. It will also be
convenient to denote by $\lambda_y : S^p \to S^p$, the map $\lambda$
evaluated at $y\in S^{4k-1}$.  Passing to homotopy classes gives a map
$\pi_{4k-1}O(p) \to \pi_{4k-1}\mathcal{C}(S^p,S^p),$ and composing
this with an isomorphism $\pi_{4k-1}\mathcal{C}(S^p,S^p) \cong
\pi_{4k-1+p}S^{p}$ then gives the $J$-homomorphism.

Recall that for $p$ sufficiently large compared to $k$, the groups
$\pi_{4k-1}O(p)$ and $\pi_{4k-1+p}S^{p}$ are independent of $p$, and
$\pi_{4k-1+p}S^{p}$ is a finite group, while $\pi_{4k-1}O(p)$ is
infinite cyclic.

Choose a map $\lambda : S^{4k-1} \to O(p)$ such that $[\lambda]\neq 0$
and $[\lambda]\in \ker J$.  This means that $\lambda$ extends to a map
\begin{equation*}
\tilde \lambda : D^{4k}_+ \to \mathcal{C}(S^{p}, S^{p}),
\end{equation*}
where $D^{4k}_+$ is a disc of radius 1.  We can assume that $\tilde
\lambda$ restricted to the collar $S^{4k-1}\times (1/2,1]\subset
  D^{4k}_+$
    coincides with the product-map $\lambda\times \mathrm{Id}$. For
    any $q\geq 4k $, we denote by $\iota : S^{p}\to S^{p}\times D^{q}$
    the inclusion $\iota: x \mapsto (x,0)$.  The map $\lambda$
    and its extension $\tilde \lambda$ give a commutative diagram
\begin{equation}\label{eq:barlambda_1a}
\xymatrix{
  S^{4k-1}\times S^p  \ar[d]^{i\times \mathrm{Id}}  
  \ar[r]^{\ \ \lambda} & S^p  \ar[d]^{\mathrm{Id}} \ar[r]^{\iota\ \ } & S^{p}\times D^{q}
\ar[d]^{\mathrm{Id}}   
\\
 D^{4k}_+\times S^p  
  \ar[r]^{\ \ \tilde\lambda} & S^p  \ar[r]^{\iota\ \ } & S^{p}\times D^{q}
}
\end{equation}
where $i$ is the inclusion of the boundary $S^{4k-1} \to D^{4k}_+$.
\begin{lemma}\label{lem:approx}
For sufficiently large $q$, we can approximate the map $\iota \circ
\tilde\lambda$ by a smoothly varying family of smooth embeddings
\begin{equation}\label{eq:hatlambda}
\hat{\lambda}_{y}: S^{p} \longrightarrow S^{p}\times D^{q}, \quad y\in D^{4k}_+,
\end{equation}
which retain the property that for $y \in S^{4k-1}\times
(1/2,1]\subset D^{4k}_+$, the maps $\hat{\lambda}_y$ agree with $\iota
  \circ \lambda_y$.
\end{lemma}
\begin{proof}
We begin by recalling that the map $\lambda:S^{4k-1} \to O(p)$ is
assumed to be smooth. By applying the Whitney approximation theorem to
$\tilde{\lambda}$ viewed as a map $\tilde{\lambda}:D^{4k}_+ \times S^p
\to S^p,$ (see for example \cite[Theorem 6.19]{Lee}), we see that by a
$C^0$ arbitrarily small homotopy relative to a boundary neighbourhood
in the domain, we can adjust this map to be smooth. Thus without loss
of generality, we may as well assume in the first place that
$\tilde{\lambda}:D^{4k}_+ \times S^p \to S^p$ is smooth.

To construct the embeddings $\hat{\lambda}_y:S^p \to S^p \times D^q,$
we first let $\eta:S^p \to D^q$ be an arbitrary embedding of the
sphere into a ball $D^q \subset {\mathbb R}^q$ centered at the origin,
and denote by $\epsilon\eta$ the composition consisting of $\eta$
followed by a scaling of $D^q$ onto itself by a factor of $\epsilon\ge
0.$ (The embedding $\eta$ clearly exists provided that $q \ge p+1$.)
Next, we introduce a {\it function} $\epsilon:D^{4k}_+ \to {\mathbb R}$,
which is identically zero in a small neighbourhood of the boundary of
$D^{4k}_+$ within the region in which $\tilde{\lambda}$ is independent
of the radial parameter, is strictly positive otherwise, and is
everywhere smooth.

Finally, set $\hat{\lambda}:D^{4k}_+ \times S^p \to S^p \times D^q$ to
be the map
\begin{equation*}\hat{\lambda}(y,x) \mapsto
  (\tilde{\lambda}_y(x),\epsilon(y)\eta(x)).
\end{equation*}
It is now immediate that this restricts to give a smoothly varying family
of smooth embeddings $\hat{\lambda}_{y}: S^{p} \longrightarrow
S^{p}\times D^{q},$ by virtue of the fact that $\eta$ is an
embedding. Moreover, these embeddings clearly agree with $\iota \circ
\tilde{\lambda}_y$ for $y$ close to the boundary of $D^{4k}_+,$ since
$\epsilon$ vanishes in this region.
\end{proof}
Denote by $N_y \to S^p$ the normal bundle of the embedding
$\hat{\lambda}_{y}$. Considering all $y\in D^{4k}_+$, we obtain a
vector bundle $N\to D^{4k}_+\times S^p$. Since $N_y \to S^p$ is a
trivial bundle, the bundle $N \to D^{4k}_+\times S^p$ is also trivial.
Then by fixing a trivialization of $N$ and using the normal
exponential map, we extend the family of embeddings
(\ref{eq:hatlambda}) to the family of embeddings
\begin{equation}\label{eq:hatlambda-1}
\bar\lambda_{y}: S^{p} \times D^q \longrightarrow S^{p}\times
D^{q}, \quad y\in D^{4k}_+.
\end{equation}
\begin{lemma}\label{lem:goette}{\rm \cite[Proposition 5.4]{GO}}
  The family of embeddings {\rm (\ref{eq:hatlambda-1})} is smoothly isotopic to a family of embeddings $ S^{p}
  \times D^p \longrightarrow S^{p}\times D^{q}$ which restricts over
  $S^{4k-1}=\p D^{4k}_+$ to give linear transformations
\begin{equation}\label{eq:barlambda}
  (x,z) \mapsto
  (\lambda_y(x), (\lambda_y^{-1}\oplus \mathrm{Id}_{\mathbb{R}^{p-q}})(z))
    \in S^{p}\times D^{q}, \ \ \ y\in S^{4k-1}.
\end{equation} 
\end{lemma}
\begin{rem}
According to Lemma \ref{lem:goette}, we can assume that the family of
embeddings {\rm (\ref{eq:hatlambda-1})} satisfies the condition {\rm
  (\ref{eq:barlambda})}. We notice also that
(\ref{eq:barlambda}) implies that $\im \bar\lambda_y$
coincides with $\{y\}\times S^p\times D^q \subset \p A_y$. In
particular, we can assume that
\begin{equation}\label{eq:A=A}
  {\mathcal A}_y=A_y \ \ \ \mbox{if $y\in
D_+^{4k}$ is near the boundary $\p D_+^{4k}$.}
\end{equation} 
\end{rem}
Now recall that the embeddings $\bar{\lambda}_y$ give rise to the
desired family of diffeomorphisms $\Lambda_y$, which by definition
coincide with the maps $\bar{\lambda}_y$ when the target space is
restricted to $\im \bar{\lambda}_y.$ This completes the construction
of the bundle $\mathcal{E}_{+}\to D^{4k}_+.$

We conclude this section with the result below, which follows from the
proof of Lemma \ref{lem:approx}:
\begin{cor}\label{cor:obs}
The diffeomorphisms $\Lambda_y:S^p \times D^q \to S^p \times D^q$ are
determined by their restriction to an arbitrarily small neighbourhood
of the sphere $S^p \times \{0\} \subset S^p \times D^q \subset \p P_y$
and its image in $\p {\mathcal A}_y,$ for each $y \in D^{4k}_+.$
\end{cor}
\begin{proof}
  We merely have to observe that in the proof of Lemma
  \ref{lem:approx} we can choose the embedding $\eta$ so that its
  image is contained in an arbitrarily small ball about the origin in
  $D^q,$ and we can choose the function $\epsilon$ to have an
  arbitrarily small upper bound.
\end{proof} 
\subsection{The Hatcher bundle $E_{\lambda}$}
It remains to describe how the bundle ${\mathcal E}_{+}\to D^{4k}_+$
is to be glued to the trivial disc bundle ${\mathcal E}_{-}\to
D^{4k}_-$ along the boundary of the base discs $\partial D^{4k}_{+}
=\partial D^{4k}_{-}=S^{4k-1}$. Recall that each disc fibre is the
union of an annulus and a `puck', and, according to
  (\ref{eq:A=A}), we can assume the annulus parts ${\mathcal A}_y$ of
the fibres are equal to $A_y$ near the boundary of $D^{4k}_+$.

We can therefore begin by gluing the annulus parts, $A_y$ and
${\mathcal A}_y$, of the fibres at the boundary of ${\mathcal E}_{+}$
and ${\mathcal E}_{-}$ via the identity map.  To glue the `puck' part of
the fibres we observe that, according to Lemma \ref{lem:goette}, the
maps $\Lambda_y: S^p\times D^q\to S^p\times D^q$ defined above are
products of rotations for each $y$ near to $\partial D^{4k}_+.$ Thus
the $\Lambda_y$ extend to diffeomorphisms $\tilde{\Lambda}_y:P_y \to
P_y$ for such $y,$ using the rotations (\ref{eq:barlambda}).  We use
the $\tilde{\Lambda}_y$ to glue the puck sub-bundles of ${\mathcal
  E}_{+}$ to ${\mathcal E}_{-}$ for $y \in S^{4k-1},$ noting that this
is consistent with the gluing of the annuli.
The disc bundle over $S^{4k}$ which results from this gluing we will
denote by ${\mathcal E}_\lambda$, since it ultimately depends on our
choice of $[\lambda] \in \ker J.$ As noted at the start of this
section, we can then double this disc bundle to produce a
desired Hatcher bundle $E_\lambda:={\mathcal
  E}_\lambda \cup {\mathcal E}_{\lambda}$ over $S^{4k}$.

It can be shown that $E_\lambda$ is bundle homeomorphic but not bundle
diffeomorphic to the corresponding trivial bundle, see
\cite[Proposition 5.8, Theorem 5.13]{GO} and \cite[Section 1]{GI-2}.

\section{The fibrewise Ricci positive metric construction}\label{Construct}
\subsection{Foreword}
In this section we will ultimately prove Main Theorem. As discussed in
section \ref{sec:Intro}, this reduces to showing that a Hatcher bundle
admits a fiberwise Ricci positive metric. Our general strategy is to
show the existence of fiberwise Ricci positive metrics on the Hatcher
disc bundles constructed in the last section, and then use the family
version of the Perelman gluing result, Theorem \ref{family}, to glue
two copies of such a disc bundle together within Ricci positivity to
create the desired object. In order to perform this gluing, we need to
consider the normal curvatures at the boundary of the disc fibres.

There is an immediate problem, however, with the boundary of the
discs: these were constructed as products $D^{p+1} \times D^q.$ Thus,
as written, each of these is a manifold with corners. Moreover, in order for these
discs to be equipped with Ricci positive metrics, it is natural to
consider product metrics which respect the topological product
structure. The resulting boundary is not smooth, however, and we need
a smooth boundary in order to apply the Perelman gluing technique.

In order to deal with this issue, our approach is to cut out a solid
`ellipsoid' from within the product of discs; see
Fig. \ref{HatcherDisk2}. This will be constructed to have a smooth boundary and normal
curvatures at the boundary (with respect to the outward normal) which are
all positive. Thus, provided the ambient metric on the product of
discs has positive Ricci curvature, we can glue two such ellipsoids
together using the Perelman gluing technique. Our main task in this
section is therefore to show how to construct such an ellipsoid, and
prove that it has the desired properties.

To avoid any confusion with indices, it is convenient
to work with the product $D^{m} \times D^n$, where the role of $m$ and
$n$ is symmetric. We specify the formulas for $m=p+1$, $n=q$ at the
end of the section when we will prove Main Theorem.

\begin{figure}[!htbp]
\begin{picture}(0,0)%
\hspace{4cm}
\includegraphics{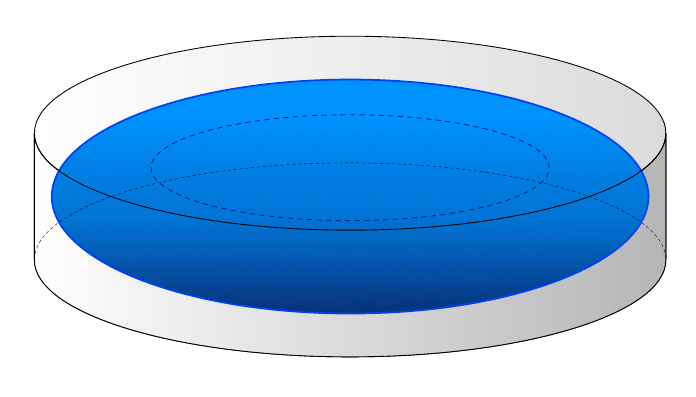}%
\end{picture}%
\setlength{\unitlength}{2500sp}%
\begingroup\makeatletter\ifx\SetFigFont\undefined%
\gdef\SetFigFont#1#2#3#4#5{%
  \reset@font\fontsize{#1}{#2pt}%
  \fontfamily{#3}\fontseries{#4}\fontshape{#5}%
  \selectfont}%
\fi\endgroup%
\begin{picture}(10500,3000)(0,0)
\end{picture}%
\caption{The solid ellipsoid contained in $D^{m}\times D^{n}$}
\label{HatcherDisk2}
\end{figure}  
\subsection{The metric on $D^{m} \times D^n$}
First, we will consider the following metric on $D^{m} \times D^n:$
\begin{equation*}
h:=ds^2+\alpha^2(s)ds^2_{m-1}+dt^2+\beta^2(t)ds^2_{n-1},
\end{equation*}
where $s$ and $t$ are the radial parameters in the discs $D^{m}$ and
$D^n$ respectively. (In our later metric constructions we will use a
slight variant of this metric.)  Let us assume that the radii of the
two discs are $s_1$ and $t_1$ respectively. We will impose the following conditions on the smooth warping
functions $\alpha$ and $\beta$:
\vspace*{-5mm}

\begin{figure}[!htbp]
\vspace{-0.5cm}
\begin{picture}(0,0)%
\includegraphics{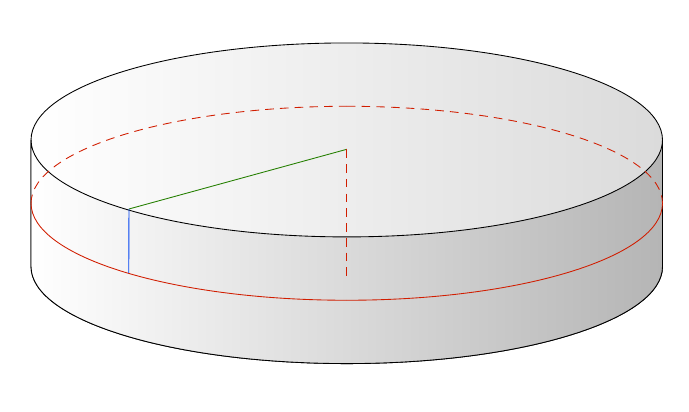}%
\end{picture}%
\setlength{\unitlength}{2500sp}%
\begingroup\makeatletter\ifx\SetFigFont\undefined%
\gdef\SetFigFont#1#2#3#4#5{%
  \reset@font\fontsize{#1}{#2pt}%
  \fontfamily{#3}\fontseries{#4}\fontshape{#5}%
  \selectfont}%
\fi\endgroup%
\begin{picture}(5000,4000)(0,0)
\put(850,1200){\makebox(0,0)[lb]{\smash{{\SetFigFont{10}{8}{\rmdefault}{\mddefault}{\updefault}{\color[rgb]{0,0,0.8}$t$}%
}}}}
\put(1700,1800){\makebox(0,0)[lb]{\smash{{\SetFigFont{10}{8}{\rmdefault}{\mddefault}{\updefault}{\color[rgb]{0,0.5,0}$s$}%
}}}}
\end{picture}%
\caption{Coordinates on the space $X=D^{m}\times D^{n}$}
\label{Hatcherdisk1}
\end{figure}
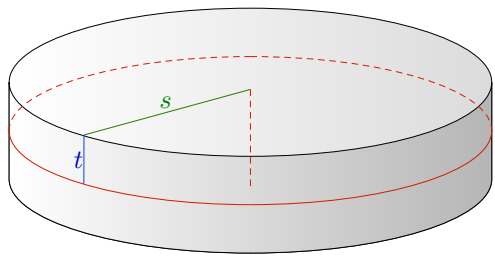  

\medskip
\begin{itemize}
\item $\alpha,$ $\beta$ are odd in a small neighbourhood of $s=0$
respectively $t=0$ (or rather, one can extend $\alpha$ and $\beta$ to
negative values of $s$ and $t$ such that this extended function is
smooth and odd), and in particular $\alpha(0)=\beta(0)=0$;
\item $\alpha'(0)=\beta'(0)=1;$
\item $\alpha'>0$ and $\beta'>0$ whenever $s$ respectively $t$ is positive;
\item $\alpha''(s)<0$ for all $s \in [0,s_1]$ and
$\beta''(t)<0$ for all $t \in [0,t_1].$
\end{itemize}
It follows easily from the warped product formulas for Ricci curvature
that these conditions ensure that the metric $h$ has strictly positive
Ricci curvature, see, for example, \cite[section 9J]{B}.
\subsection{Specifying the ellipsoid}
In order to construct the ellipsoid, we introduce a unit speed curve
$\mu=\mu(r)$ into the $(s-t)$-plane. This curve will have the
profile given in Fig. \ref{funcf}.

Notice that the illustrated curve separates the rectangle
$[0,s_1] \times [0,t_1]$ into two regions, and suppose that the
parameter $r$ is such that $\mu(0)=(0,t_0)$ and $\mu(r_0)=(s_0,0)$ for
some $s_0 \in (0,s_1)$ and $t_0 \in (0,t_1).$ We will define the
ellipsoid $\mathbf{E}$ to be the subset of $D^{p+1} \times D^q$
consisting of all elements whose $s$ and $t$ coordinates lie in the
region on or below this curve.
 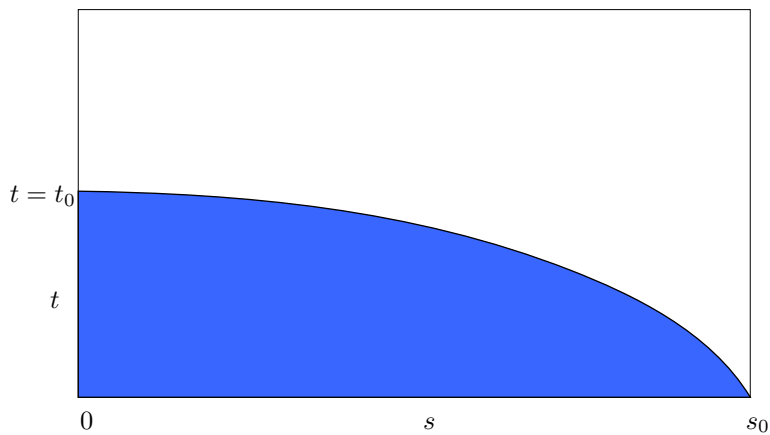
\begin{figure}[!htbp]
\begin{picture}(0,0)%
\includegraphics{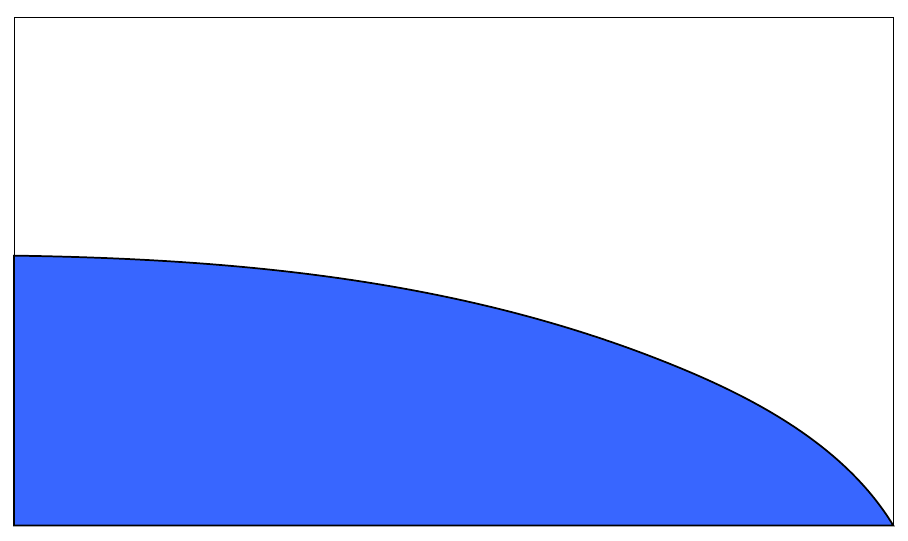}%
\end{picture}%
\setlength{\unitlength}{2500sp}%
\begingroup\makeatletter\ifx\SetFigFont\undefined%
\gdef\SetFigFont#1#2#3#4#5{%
  \reset@font\fontsize{#1}{#2pt}%
  \fontfamily{#3}\fontseries{#4}\fontshape{#5}%
  \selectfont}%
\fi\endgroup%
\begin{picture}(7000,4000)(0,0)
\put(100,-200){\makebox(0,0)[lb]{\smash{{\SetFigFont{10}{8}{\rmdefault}{\mddefault}{\updefault}{\color[rgb]{0,0,0}$0$}%
}}}}
\put(6700,-200){\makebox(0,0)[lb]{\smash{{\SetFigFont{10}{8}{\rmdefault}{\mddefault}{\updefault}{\color[rgb]{0,0,0}$s_0$}%
}}}}

\put(3500,-200){\makebox(0,0)[lb]{\smash{{\SetFigFont{10}{8}{\rmdefault}{\mddefault}{\updefault}{\color[rgb]{0,0,0}$s$}%
}}}}
\put(-600,2050){\makebox(0,0)[lb]{\smash{{\SetFigFont{10}{8}{\rmdefault}{\mddefault}{\updefault}{\color[rgb]{0,0,0}$t=t_0$}%
}}}}
\put(-200,1000){\makebox(0,0)[lb]{\smash{{\SetFigFont{10}{8}{\rmdefault}{\mddefault}{\updefault}{\color[rgb]{0,0,0}$t$}%
}}}}
\end{picture}%
\caption{The curve $\mu$ which gives rise to the ellipsoid $\mathbf{E}$}
\label{funcf}
\end{figure}   

We need to specify $\mu$ in more detail, and to this end we will write
$\mu(r)=(\mu_s(r),\mu_t(r)).$ Let us impose the following conditions
on $\mu_s,\mu_t:$
\begin{enumerate}
\item $\mu_s(0)=0,$ $\mu_s(r_0)=s_0,$ $\mu_s'(0)=1,$ $\mu_s'(r_0)=0,
$ $\mu_s''(r)<0$ for all $r \in [0,r_0],$ $\mu_s$ is odd in a
neighbourhood of $r=0$, and even locally about the point $r=r_0$ (in
the sense that there is a smooth extension such that
$\mu_s(r_0-\epsilon)=\mu_s(r_0+\epsilon)$ for all sufficiently small
$\epsilon>0$);
\item $\mu_t(0)=t_0,$ $\mu_t(r_0)=0,$ $\mu_t'(0)=0,$ $\mu_t'(r_0)=-1,$
$\mu_t''(r)<0$ for all $r \in [0,r_0],$ $\mu_t$ is even about the
point $r=0,$ and odd in a neighbourhood of $r=r_0.$
\end{enumerate}
Given a unit speed curve $\mu$ satisfying (1) and (2) above, we need
to check that the resulting ellipsoid has a smooth boundary. Firstly,
it is clear from the smoothness of all the functions involved that
this boundary will indeed be smooth everywhere except possibly
when $r=0$ or $r=r_0.$ We must therefore check the corresponding
`ends' of the ellipsoid for smoothness. With this in mind, we begin by
observing that the metric induced by $h$ on the ellipsoid is
\begin{equation}
h_{\mathbf{E}}:=dr^2+\alpha^2(\mu_s(r))ds^2_{m-1}+\beta^2(\mu_t(r))ds^2_{n-1}.
\end{equation}
Given that $\mu_s(0)=0, \mu_t(r_0)=0,$ and that $\mu_s,\mu_t>0$
otherwise, we see immediately from the form of the metric,
$h_{\mathbf{E}}$, that this is a (not necessarily smooth) metric on a
sphere of dimension $m+n-1.$ The boundary conditions which such a
metric must satisfy in order to give a smooth sphere metric are
well-known (see for example \cite[Section 1.4]{PP}): the scaling
funtions $\alpha(\mu_s(r))$ and $\beta(\mu_t(r))$ must obey the
following rules
\begin{enumerate}
\item[(i)] be everywhere non-negative, with $\alpha(\mu_s(r))=0$ if
  and only if $r=0$, and $\beta(\mu_t(r))=0$ if and only if $r=r_0$;
\item[(ii)] $\alpha(\mu_s(r))$ must be odd at $r=0$ and even at $r=r_0$;
\item[(iii)] $\beta(\mu_t(r))$ must be even at $r=0$ and odd at $r=r_0$;
\item[(iv)] the derivative of $\alpha(\mu_s(r))$ must take the value 1
  at $r=0$, and that of $\beta(\mu_t(r))$ must take the value $-1$ at
  $r=r_0.$
\end{enumerate}
Property (i) follows immediately from the conditions imposed on
$\alpha,\beta, \mu_s, \mu_t.$ For property (iii) we note that by
definition $\mu_t(r)$ is odd at $r=r_0$ and $\beta(t)$ is odd at
$t=0$, and it follows trivially from this that the composition
$\beta(\mu_t(r))$ is odd at $r=r_0.$ For the evenness requirement it
suffices to note that the composition of an even function followed by
an arbitrary function is trivially even. Property (ii) follows by
similar arguments. Finally, property (iv) follows by the chain rule
since $\alpha'(0)\mu_s'(0)=1,$ and $\beta'(0)\mu_t'(r_0)=-1.$

In summary then, we have demonstrated how to choose a unit speed curve
$\mu$ such that the resulting ellipsoid is smooth, and we will work
with the same subset of $D^{p+1} \times D^q$ for each fibre of the
Hatcher disc bundle when we construct the fiberwise metric later in
this section.
\subsection{Normal curvatures of the ellipsoid}
The other issue we need to address in relation to the Perelman gluing
of discs (or rather ellipsoids) is that of the normal curvatures at
the boundary. As observed previously (Corollary \ref{pcor}), it will
suffice if these normal curvatures are all positive with respect to
the outward pointing normal. It turns out, however, that the normal
curvatures of the ellipsoid we have constructed are only non-negative
with respect to the ambient metric $h$. In particular, the normal
curvatures vanish at the points of the ellipsoid corresponding to
$r=0$ and $r=r_0.$ To rectify this situation we work with the same
ellipsoid, but a slightly modified metric on $D^{m} \times D^n$.

Let us define a metric $g$ on $D^{m} \times D^n$ as follows:
\begin{equation}
g:=\delta^2(t)ds^2+\delta^2(t)\alpha^2(s)ds^2_{m-1}+
\gamma^2(s)dt^2+\gamma^2(s)\beta^2(t)ds^2_{n-1}.
\end{equation}
The new functions introduced here, $\delta(t)$ and $\gamma(s)$, are
chosen so as to satisfy the following properties:
\begin{enumerate}
\item[(a)] $\delta'(t) \ge 0$ for all $t \in [0,t_1],$
  $\delta'(t_0)>0,$ $\delta(t)\equiv 1$ in a neighbourhood of $t=0;$
\item[(b)] $\gamma'(s) \ge 0$ for all $s \in [0,s_1],$
  $\gamma'(s_0)>0,$ $\gamma(s)\equiv 1$ in a neighbourhood of $s=0.$
\end{enumerate}
We will see that the positivity of the derivatives of $\delta$ and
$\gamma$ at $t_0$, respectively $s_0$, is enough to give us strictly
positive normal curvatures globally. Of course we must not forget that
the metric $g$ must have positive Ricci curvature. By the openness of
the positivity condition we can choose $\delta$ and $\gamma$
satisfying (a) and (b) above sufficiently close in a $C^2$-sense to
the constant function with value 1 so that $\Ric(g)>0.$ We therefore
add a third condition:
\begin{enumerate}
\item[(c)] $\delta$ and $\gamma$ are such that $\Ric(g)>0,$ at least
  in some neighbourhood of $\mathbf{E}$.
\end{enumerate}
\begin{lemma}\label{lem:normal-curv}
The normal curvatures at the boundary of the ellipsoid $\mathbf{E}$
are all stricly positive with respect to the ambient metric $g$.
\end{lemma}
\begin{proof}
%We now need to establish the claim that with respect to the ambient
%metric $g,$ the normal curvatures at the boundary of $\mathbf{E}$ are
%all stricly positive.  
We work locally, and begin by fixing a point
$x_1=(s_1,a_1,t_1,b_1) \in \mathbf{E}$, where $a_1 \in S^{m-1}$ and
$b_1 \in S^{n-1}.$ About the points $a_1$ and $b_1$, introduce normal
coordinate systems locally into $S^{m-1}$ and $S^{n-1}.$ Together with
the $s$ and $t$ coordinates, these combine to give a local coordinate
system in $D^{m} \times D^n.$ With respect to these coordinates we can
represent $g$ by the block-diagonal matrix
\begin{equation*}
 \begin{split}
 g = \begin{bmatrix} \delta^2(t) & & & \\ & \gamma^2(s) & & \\ & &
   \delta^2(t)\alpha^2(s)A_{m-1} & \\ & & &
   \gamma^2(s)\beta^2(t)B_{n-1}
 \end{bmatrix},
 \end{split}
\end{equation*}
where $A_{m-1}$ and $B_{n-1}$ represent $ds^2_{n-1}$ respectively
$ds^2_{m-1}$ with respect to the chosen normal coordinate systems on
the spheres. Note that at the the points $a_1$ and $b_1$, $A_{m-1}$
and $B_{n-1}$ are both identity matrices and have vanishing first
derivatives. Hence at the point $x_1$ we have $g_{ij}\neq 0$ if and
only if $i=j,$ $g^{ii}=\frac{1}{g_{ii}},$ and the derivatives
$g_{ij,k}=0$ whenever $k$ is a direction tangent to $S^{m-1}$ or
$S^{n-1}.$ We will assume that all computations below are carried out
at this point.

Using the formula
\begin{equation*}
  \Gamma_{ij}^{k}=\frac{1}{2}g^{kl}(g_{il,j}+g_{jl,i}-g_{ij,l}),
\end{equation*}
it is straightforward to compute the corresponding Christoffel
symbols. The list below consists of precisely those Christoffel
symbols which are non-zero. Beginning with the case when each of the
indices $i,j$ and $k$ are $s$ or $t$, we have the following.
\begin{equation*}
  \Gamma^s_{st}=\Gamma_{ts}^{s}=\frac{\delta'(t)}{\delta(t)},\quad \Gamma^s_{tt}=
  \frac{-\gamma'(s)\gamma(s)}{\delta^{2}(t)},\quad \Gamma^t_{ss}=
  \frac{-\delta(t)\delta'(t)}{\gamma^{2}(s)},\quad \Gamma^t_{st}=
  \Gamma^t_{ts}= \frac{ \gamma'(s)}{\gamma(s)}.
\end{equation*}
Then, using the symbols $a$ and $b$ to represent any of the coordinate
functions on $S^{m-1}$ or $S^{n-1}$ respectively, we list the
remaining non-zero Christoffel symbols.
\begin{equation*}
\Gamma^s_{aa}=-\alpha(s)\alpha'(s), \quad
\Gamma^a_{sa}=\Gamma_{as}^{a}=\frac{\alpha'(s)}{\alpha(s)},\quad 
\Gamma^t_{aa}=\frac{-\delta(t)\delta'(t)\alpha^{2}(s)}{\gamma^{2}(s)},\quad
\Gamma^a_{ta}=\Gamma_{at}^{a}=\frac{\delta'(t)}{\delta(t)},
\end{equation*}
\begin{equation*}
\Gamma^s_{bb}=\frac{-\gamma(s)\gamma'(s)\beta^{2}(t)}{\delta^{2}(t)}, \quad
\Gamma^b_{sb}=\Gamma_{bs}^{b}=\frac{\gamma'(s)}{\gamma(s)},\quad 
\Gamma^t_{bb}=-\beta(t)\beta'(t),\quad
\Gamma^b_{tb}=\Gamma_{bt}^{b}=\frac{\beta'(t)}{\beta(t)}.
\end{equation*}
From this we compute certain covariant derivatives involving
coordinate vector fields, $\p_s, \p_t, \p_a$ and $\p_b$, which will we
will make use of shortly. In particular, we see that at the point
$x_1$ we have
\begin{equation*}
\nabla_{\p_s}{\p_s}=\frac{-\delta(t)\delta'(t)}{\gamma^{2}(s)}\p_{t},\quad
\nabla_{\p_t}{\p_t}=\frac{-\gamma'(s)\gamma(s)}{\delta^{2}(t)}\p_{s},\quad
\nabla_{\p_s}{\p_t}=\nabla_{\p_t}{\p_s}=\frac{\delta'(t)}{\delta(t)}\p_{s}+\frac{ \gamma'(s)}{\gamma(s)}\p_{t},
\end{equation*}
\begin{equation*}
\nabla_{\p_t}{\p_a}=\nabla_{\p_a}{\p_t}=
\frac{\delta'}{\delta}\p_a, \quad \nabla_{\p_s}{\p_a}=\nabla_{\p_a}{\p_s}=
\frac{\alpha'}{\alpha}\p_a,
\end{equation*}
\begin{equation*}
\nabla_{\p_t}{\p_b}=\nabla_{\p_b}{\p_t}=\frac{\beta'}{\beta}\p_b, \quad \nabla_{\p_s}{\p_b}=\nabla_{\p_b}{\p_s}=\frac{\gamma'}{\gamma}\p_b,
\end{equation*}
\begin{equation*}
\nabla_{\p_a}{\p_a}=-\frac{\delta'\delta\alpha^2}{\gamma^2}\p_t-\alpha'\alpha\p_s, \quad \nabla_{\p_b}{\p_b}=-\frac{\gamma'\gamma\beta^2}{\delta^2}\p_s-\beta'\beta\p_t,
\end{equation*}
and $\nabla_{\p_a}{\p_b}=\nabla_{\p_b}{\p_a}=0.$

The statement that all normal curvatures are positive is clearly
equivalent to the statement that the second fundamental form is
positive definite. We will compute second fundamental forms, and will
break up the computation into directions tangent to $S^{m-1}$,
$S^{n-1}$, and tangent to the curve $\mu$. Notice that $\mu'(r)$ is
everywhere tangent to the boundary of the ellipsoid, and this
direction is orthogonal (with respect to $g$) to both $S^{m-1}$ and
$S^{n-1}.$ Explicitly we have
$T(r):=\mu'(r)=\mu_s'(r)\partial_s+\mu_t'(r)\partial_t.$ It is easy to
see that the outward normal vector to the ellipsoid lies in the
$(s-t)$-plane. If we represent it as $N=c_s \partial_s
+c_t \partial_t$ then it is clear that the coefficients $c_s, c_t$ are
functions of $r$. Moreover, it is evident from our choice of $\mu$
that $c_s(r_0)=0,$ $c_t(0)=0,$ and that $c_s,c_t>0$ otherwise.

The second fundamental form $II(u,v)$ is defined by
$II(u,v)=-g(\nabla_u v,N).$ Thus in order to show positive
definiteness it suffices to establish that the components of $\nabla_u
u$ in the $\partial_s$ and $\partial_t$ directions are non-positive,
at least one of the coefficients is negative for all $r \in (0,r_0),$
at $r=0$ (where $c_t=0$) we need the coefficient of $\partial_s$ to be
negative, and at $r=r_0$ (where $c_s=0$) we need the coefficient of
$\partial_t$ to be negative. (Of course if $u \in TS^{m-1}$ then we
must automatically have $r>0$ else this sphere is not defined, and
similarly we need $r<r_0$ if $u \in TS^{n-1}$.)

Consider first $\partial_a \in TS^{m-1}.$ From the covariant
derivative expressions above we observe that the coefficient of
$\partial_s$, namely $-\alpha'\alpha,$ is non-positive and strictly
negative for all $r \in (0,r_0),$ however it vanishes at $r=r_0$. (We
have $r>0$ in order for the vector $\p_a$ to make sense, as noted
above.) The coefficient of $\partial_t$ is
$-\delta'\delta\alpha^2\gamma^{-2},$ and this is clearly non-negative,
but negative at $r=r_0$ since $\delta'(t_0)>0$ by definition. Thus we
have $II(\p_a,\p_a)<0$ as required. Analogous arguments apply for
$II(\p_b,\p_b).$

Next, we investigate $\nabla_T T.$ We have
\begin{align*}
  \nabla_T T=&\mu'_s(\p_s \mu'_s)\p_s+{\mu'_s}^2\nabla_{\p_s}{\p_s}+
  \mu'_s(\p_s\mu'_t)\p_t+\mu'_s\mu'_t\nabla_{\p_s}{\p_t} \\
  & +\mu'_t(\p_t \mu'_s)\p_s+\mu'_t\mu'_s\nabla_{\p_t}{\p_s}+
  \mu'_t(\p_t\mu'_t)\p_t+{\mu'_t}^2\nabla_{\p_t}{\p_t}. 
\end{align*}
In order to simplify this expression, we note that by definition of
$\mu,$ the coordinate functions $\mu_s(r)$ and $\mu_t(r)$ are
one-to-one, and therefore invertible. Viewing $s$ as a function of $r$
along $\mu$ we clearly have $s(r)=\mu_s(r),$ and hence
$r(s)=\mu^{-1}_s(s).$ Differentiating with respect to $s$ then yields
\begin{align*}
  \p_s\mu_s'(r)&=\p_s\mu_s'(\mu^{-1}_s(s)) \\
&=\mu_s''(\mu_s^{-1}(s))\frac{1}{\mu'_s(\mu^{-1}(s))} \\
&=\mu''_s(r)/\mu_s'(r). 
\end{align*}
Analogous computations give 
\begin{equation}\label{derivs1}
  \p_s \mu'_t=\mu''_t/\mu'_s, \quad \p_t\mu'_s=\mu_s''/\mu_t',
  \quad \p_t\mu_t'=\mu_t''/\mu_t'.
\end{equation}
It follows immediately that
\begin{equation}\label{derivs2}
  \mu'_s(\p_s \mu'_s)=\mu_s'',
  \quad \mu'_s(\p_s \mu'_t)=\mu_t'',
  \quad \mu'_t(\p_t \mu_s')=\mu_s'',
  \quad \mu_t'(\p_t \mu_t')=\mu_t''.
\end{equation}
Notice that for the above calucations to be valid as stated, we must
assume that $\mu'_s,\mu'_t \neq 0.$ This is fine precisely when $r \in
(0,r_0).$ However, observe that the right-hand sides of the
expressions (\ref{derivs2}) are defined for all $r \in [0,r_0],$ and
we can infer from this that the limits as $r \to 0^+$ and $r \to
r_0^-$ in (\ref{derivs1}) must be well-behaved.

We can now use the above calculations to simplify the expression for
$\nabla_T T:$
\begin{align*}
  \nabla_T T=&\p_s\Bigl(2\mu_s''-{\mu'_s}^2\frac{\delta'\delta}{\gamma^2}+
  2\mu'_s\mu'_t\frac{\delta'}{\delta}\p_s\Bigr) \\
  & +\p_t\Bigl(2\mu_t''-{\mu_t'}^2\frac{\gamma'\gamma}{\delta^2}+
  2\mu'_s\mu'_t\frac{\gamma'}{\gamma}\Bigr). 
\end{align*}
In each of the above brackets, notice that the terms are negative,
non-positive and non-positive respectively. It follows that
$II(T,T)>0$ as required.

It remains, then to consider `mixed' terms, that is $II(v,w)$ where
$v$ and $w$ belong to two of the three basic directions in $\p E$,
namely $TS^{m-1},$ $TS^{n-1}$ and $Span\{T\}$.  We see from the
covariant derivative expressions above that $\nabla_v w=\nabla_w v=0$
if $v \in TS^{m-1}$ and $w \in TS^{n-1},$ and hence in this case we
have $II(v,w)=0.$ For this $v$ and $w$ we also have
$\nabla_{\partial_t} v=\frac{\delta'}{\delta}v$ and
$\nabla_{\partial_t} w=\frac{\delta'}{\delta}w,$ and so
$II(\partial_t,v)=-\frac{\delta'}{\delta}g(v,N)=0.$ Similarly with
$\partial_s$ in place of $\partial_t$. Thus the mixed terms of $II$
all vanish, and hence we can conclude that $II$ is positive definite,
as required.
\end{proof}
Let us summarise the above constructions:
\begin{prop} 
There is a Ricci positive metric $g$ on $D^{m} \times D^n$ and a
codimension zero solid ellipsoid $\mathbf{E} \subset D^{m} \times D^n$
such that $\partial \mathbf{E}$ is a smooth submanifold of $D^{m}
\times D^n$ and the normal curvatures of $\partial \mathbf{E}$ (with
respect to the outward pointing normal) are all positive.
\end{prop}
\subsection{Proof of the Main Theorem}
Recall from section \ref{sec:Intro} that to establish the theorem it
suffices to construct a fiberwise Ricci positive metric on each
Hatcher sphere bundle. In order to do this, we will begin by
reconsidering the construction of the Hatcher disc bundle from section
\ref{Hatch}.

Now we switch to the relevant notations, i.e., $m=p+1$ and $n=q$. For
each point $y \in D^{4k}_+$ we have
\begin{equation*}
D^{p+1} \times D^q =P_y
\cup_{id} A_y \stackrel{\Phi_y}{\cong} P_y \cup_{\Lambda_y} {\mathcal
  A}_y,
\end{equation*}
where we refer the reader to section \ref{Hatch} for the notation. The
ellipsoid $\mathbf{E}$ is a subset of $D^{p+1} \times D^q$, and so for
each $y \in D^{4k}_+$ there is an ellipsoid 
\begin{equation*}
\mathbf{E}_y:=\Phi_y(\mathbf{E})
\subset P_y \cup_{\Lambda_y} {\mathcal A}_y.
\end{equation*}
Collectively, these ellipsoid fibres form a sub-bundle ${\boldsymbol
  {\mathcal E}}^{ell}_{+}$ of ${\mathcal E}_{+}.$ Pushing forward the
metric $g$ via $\Phi_y$ and restricting to $\mathbf{E}_y$ equips each
$\mathbf{E}_y$ with a Ricci positive metric with positive normal
curvatures (with respect to the outward normal) at the boundary.
Moreover as $y$ varies across $D^{4k}_+,$ we obtain in this way a
smoothly varying family of fibre metrics on ${\boldsymbol {\mathcal
 E}}^{ell}_{+}.$

We similarly form a product bundle ${\boldsymbol {\mathcal
    E}}^{ell}_{-} \to D^{4k}_{-}$ with total space $D^{4k}_{-} \times
\mathbf{E}$, and take the obvious fiberwise metric where each fibre
$\mathbf E$ is equipped with the metric induced by $g$.
For each fibre $\mathbf{E}_y \subset {\boldsymbol {\mathcal
    E}}^{ell}_{+}$, notice that we have a
decomposition
\begin{equation*}
\mathbf{E}_y=(\mathbf{E}_y \cap P_y) \cup
(\mathbf{E}_y \cap {\mathcal A}_y),
\end{equation*}
and similarly for the fibres of
${\boldsymbol {\mathcal E}}^{ell}_{-}.$

In order to form the Hatcher disc bundle, we need to glue the bundles
${\boldsymbol {\mathcal E}}^{ell}_{+}$ and ${\boldsymbol {\mathcal
    E}}^{ell}_{-}$ along the boundaries of their base discs. The
procedure for gluing the `full' disc bundles ${\mathcal E}_{+}$ and
${\mathcal E}_{-}$ is described at the end of section
\ref{Hatch}. Recall that for each pair of fibres in ${\mathcal E}_{+}$
and ${\mathcal E}_{-}$ being identified, the annulus parts are
identified via the identity map, but the inner `puck' regions are
identified using diffeomorphisms $\tilde{\Lambda}_y: P_y \to P_y,$
which by Lemma \ref{lem:goette} split as a product of rotations on the
two disc factors. Before proceeding further, we note that these gluing
maps restrict to give gluing maps between ${\boldsymbol {\mathcal
    E}}^{ell}_{+}$ and ${\boldsymbol {\mathcal E}}^{ell}_{-},$ since
the annulus and puck parts of the respective ellipsoid bundles agree
near the boundary of the base discs, and are invariant under rotation
of the factors. Note further that by Corollary \ref{cor:obs} in
section \ref{Hatch}, we do not lose any gluing information by reducing
the fibres from the original product of discs considered in section
\ref{Hatch} to the ellipsoids currently under consideration. Thus the
bundle we will construct using ${\boldsymbol {\mathcal E}}^{ell}_{+}$
and ${\boldsymbol {\mathcal E}}^{ell}_{-}$ will be diffeomorphic to
that formed from ${\mathcal E}_{+}$ and ${\mathcal E}_{-}.$

From a metric perspective, let us focus first on the puck
sub-bundles. As $\Phi_y$ is the identity mapping on $P_y$, the puck
sub-bundle within ${\boldsymbol {\mathcal E}}^{ell}_{+}$ is just a
product, with each fibre equipped with the restriction of $g$. Now the
metric $g$ displays rotational symmetry with respect to both disc
factors, and so pulling-back $g|_{P_y}$ via the map
$\tilde{\Lambda}_y$ results in a metric identical to $g|_{P_y}.$ Since
we have set things up so that the metrics near the boundaries of both
${\mathcal E}_{+}$ and ${\mathcal E}_{-}$ are independent of the
radial parameter in the base, we see that gluing the puck sub-bundles
along $S^{4k-1}$ in this way yields a smooth fiberwise metric. (It is
perhaps worth remarking that if we were trying to construct a
submersion metric on the {\it whole} Hatcher disc bundle - as opposed
to creating a mere fiberwise metric - then the twisting involved in
gluing the bundles ${\mathcal E}_{+}$ and ${\mathcal E}_{-}$ would
have non-trivial metric implications in directions transverse to the
fibres.)

Turning our attention to the gluing of the annular regions, we
similarly observe that the metric on the annuli close to the boundary
of ${\mathcal E}_{+}$ is a push-forward via $\Phi_y$ of the
rotationally symmetric metric $g|_{A_y}.$ Although $\Phi_y$ acts
non-trivially on the annuli, it nevertheless acts by rotation in both
$S^{m-1}$ and $S^{n-1}$ directions for $y$ close to $\p D^{4k}_{+}.$
Thus the pull-back metric on the annuli is identical to the original
over the boundary of the base disk, and so gluing the annular part of
${\mathcal E}_{+}$ to ${\mathcal E}_{-}$ via the identity creates a
smooth fiberwise metric in the annular region also.

In summary, we have created a smooth fibrewise Ricci positive metric
on the fibres of the Hatcher disc bundle ${\mathcal E}={\mathcal E}_+
\cup {\mathcal E}_{-}.$ It is immediate that restricting everything in
the above argument to the ellipsoid sub-bundles ${\boldsymbol
  {\mathcal E}^{ell}}_{-}$ and ${\boldsymbol {\mathcal E}}^{ell}_{+}$
creates a fiberwise Ricci positive metric on the ellipsoid sub-bundle
of the Hatcher disc bundle ${\boldsymbol {\mathcal E}}^{ell} \subset
{\mathcal E}$, with the normal curvatures at the boundary of each
fibre being positive with respect to the outer normal.

Finally, we wish to glue two identical copies of the Hatcher disc
bundle ${\boldsymbol {\mathcal E}}^{ell}$ equipped with the above
fiberwise metric so as to construct the desired Hatcher sphere
bundle. Metrically this is now possible using the family gluing result,
Theorem \ref{family}, as a
consequence of the positive normal curvatures at the boundary. We thus
create a Hatcher sphere bundle with a smooth fibrewise Ricci positive
metric, as required to establish the theorem. \hfill $\square$

%%Bibliog

\end{document}